\documentclass{article}

\usepackage{amsmath,amsfonts,latexsym,amssymb,amsthm,url, titling}

\newtheorem{theorem}{Theorem}[section]
\newtheorem{proposition}[theorem]{Proposition}
\newtheorem{corollary}[theorem]{Corollary}
\newtheorem{lemma}[theorem]{Lemma}
\newtheorem{remark}[theorem]{Remark}

\newtheorem{example}[theorem]{Example}
\newtheorem{conjecture}[theorem]{Conjecture}

\DeclareMathOperator{\cond}{cond}

\DeclareMathOperator{\height}{h}
\DeclareMathOperator{\lgcd}{lgcd}
\DeclareMathOperator{\gcds}{{{\textstyle \gcd_S}}}

\newcommand{\Q}{\mathbb{Q}}

\newcommand{\Z}{\mathbb{Z}}

\newcommand{\II}{\mathcal{I}}
\newcommand{\NN}{\mathcal{N}}
\newcommand{\OO}{\mathcal{O}}

\newcommand{\gerp}{\mathfrak{p}}

\newcommand{\ugamma}{{\underline{\gamma}}}
\newcommand{\udelta}{{\underline{\delta}}}

\newcommand{\ub}{{\underline{b}}}
\newcommand{\uc}{{\underline{c}}}
\newcommand{\um}{{\underline{m}}}
\newcommand{\un}{{\underline{n}}}
\newcommand{\ut}{{\underline{t}}}
\newcommand{\uu}{{\underline{u}}}
\newcommand{\ux}{{\underline{x}}}
\newcommand{\uy}{{\underline{y}}}
\newcommand{\uz}{{\underline{z}}}
\newcommand{\uw}{{\underline{w}}}

\newcommand{\eps}{\varepsilon}
\newcommand{\ph}{\varphi}

\title{Multiplicative dependence in the sumset of  multiplicative groups}

\author{Yuri~Bilu, Florian~Luca}

\date{\today}

\numberwithin{equation}{section}

\makeatletter

\renewcommand*\l@section[2]{
	\ifnum \c@tocdepth >\z@
	\addpenalty\@secpenalty
	\addvspace{0.2em \@plus\p@}
	\setlength\@tempdima{1.5em}
	\begingroup
	\parindent \z@ \rightskip \@pnumwidth
	\parfillskip -\@pnumwidth
	\leavevmode \bfseries
	\advance\leftskip\@tempdima
	\hskip -\leftskip
	#1\nobreak\hfil \nobreak\hb@xt@\@pnumwidth{\hss #2}\par
	\endgroup
	\fi}

\makeatother

2

\begin{document}
\hfuzz=4pt

\maketitle

\begin{abstract}
Let~$\Gamma$ and~$\Delta$ be finitely generated multiplicative groups of algebraic numbers such that ${\Gamma\cap\Delta}$ is a finite group. We show that, up to finitely many exceptions, non-zero sums ${x_1+y_1}$ and ${x_2+y_2}$, with ${x_1, x_2\in \Gamma}$ and ${y_1,y_2\in \Delta}$, are multiplicatively dependent only if ${x_1/x_2=y_1/y_2}$ is a root of unity. For ${m\ge 3}$, we discuss possible shapes of~$m$ multiplicatively dependent sums ${x_1+y_1, \ \ldots, \ x_m+y_m}$ with ${x_1, \ldots, x_m \in \Gamma}$ and ${y_1, \ldots, y_m \in \Delta}$.  For ${m=3}$ we classify such sums, up to finitely many exceptions, assuming the $abc$-conjecture. 
\end{abstract}

 {\footnotesize
 
\tableofcontents

}

\section*{Preface (by Yuri Bilu)}

Mu coauthor and friend Florian Luca died suddenly during the final preparation of this manuscript. He had no chance to proofread it. I assume full responsibility for all the remaining anomalies. 

\section{Introduction}
\label{sintro}

Let~$\Gamma$ and~$\Delta$ be  finitely generated subgroups of~$\bar\Q^\times$. Call them \textit{almost disjoint} if the intersection $\Gamma\cap\Delta$ is finite.  We are interested in the following question: how often non-zero elements of the sumset 
$$
\Gamma+\Delta:=\{x+y:  x\in \Gamma, y \in \Delta\}
$$
can be multiplicatively dependent?  To make it more precise, for every positive integer~$m$ we want to classify, up to finitely many exceptions, possible  ${x_1, \ldots, x_m\in \Gamma}$ and ${y_1, \ldots, y_m\in \Delta}$ such that the~$m$ sums 
$
{x_1+y_1,\ \ldots,\ x_m+y_m}
$ 
are all non-zero and multiplicatively dependent.

For ${m=1}$ this reduces to the following question: find all ${x\in \Gamma}$ and ${y\in \Delta}$ such that ${x+y}$ is a root of unity. This is the classical \textit{binary unit equation},  and the result (essentially, due to Baker) is well-known.

\begin{theorem}
\label{thmone}
There exist at most finitely many couples ${(x,y)\in \Gamma\times \Delta}$ such that ${x+y}$ is a root of unity, and all of them can be effectively determined.
\end{theorem}

Note that in this case we do not need~$\Gamma$ and~$\Delta$ to be almost disjoint. We  recall the proof in Subsection~\ref{sslocalh}, see Theorem~\ref{thunits} therein. 

\subsection{Results}

Already for ${m=2}$ there is an infinite family of multiplicatively dependent cases. Let ${x\in \Gamma}$ and ${y\in \Delta}$ be such that ${x+y\ne 0}$, and let~$\xi$ be a root of unity belonging to both~$\Gamma$ and~$\Delta$. Then the sums ${x+y}$ and ${\xi x+\xi y}$ are multiplicatively dependent. Our first principal result is the following theorem, which confirms that, up to finitely many exceptions, these are the only examples.

\begin{theorem} 
\label{thmtwo}
Let~$\Gamma$ and~$\Delta$ be almost disjoint finitely generated multiplicative groups of algebraic numbers. 
Let ${x_1, x_2\in \Gamma}$ and ${y_1, y_2\in \Delta}$ be such that the sums ${x_1+y_1}$ and ${x_2+y_2}$ are non-zero  and multiplicatively dependent. Then, with finitely many exceptions, ${x_1/x_2=y_1/y_2}$ is a root of unity. 
\end{theorem}

Unlike Theorem~\ref{thmone}, in Theorem~\ref{thmtwo} the hypothesis of almost disjointness cannot be dropped, as the following example shows.

\begin{example}
\label{expothree}
Set ${\Gamma:=\langle -1, 3\rangle}$ and ${\Delta:=\langle 2,3\rangle}$. Then ${\Gamma+\Delta}$ contains all the powers of~$3$. 
\end{example}

Theorem~\ref{thmtwo} is non-effective in general: we do not bound effectively the heights of the ``finitely many exceptions''. However, it is effective in the case when both~$\Gamma$ and~$\Delta$ are of rank~$1$; see Remark~\ref{reff} for the details. 




When ${m=3}$, two new infinite families of multiplicatively dependent sums emerge. 
Let ${x\in \Gamma}$ and ${y\in \Delta}$ be such that ${x^2\ne y^2}$. Then 
\begin{equation}
\label{efamthreeone}
x-y, \quad x+y, \quad x^2-y^2
\end{equation}
are multiplicatively dependent. 


To introduce another infinite family of three multiplicatively dependent sums, let us make the following observation: if ${x_1, x_2\in \Gamma}$ and ${y_1,y_2\in \Delta}$ are such that ${x_1+y_1}$ and ${x_2y_2}$ are multiplicatively dependent, then the three sums 
\begin{equation}
\label{ethreesumsminusone}
x_1+y_1, \ x_2+y_2, \ x_2^{-1}+y_2^{-1} 
\end{equation}
are multiplicatively dependent. Moreover,  if ${x_1+y_1}$ and ${x_2y_2}$ are multiplicatively dependent, then so are ${x_1+y_1}$ and ${x_2^\ell y_2^\ell}$ for any positive integer~$\ell$. If $x_2y_2$ is not a root of unity\footnote{By almost disjointness, it is equivalent to saying that at least one of them is not a root of unity.}, then we obtain an infinite family 
\begin{equation}
\label{ethreesumsminusonell}
x_1+y_1, \ x_2^\ell+y_2^\ell, \ x_2^{-\ell}+y_2^{-\ell} 
\end{equation}
of three multiplicatively dependent sums in ${\Gamma+\Delta}$. 

It is not completely obvious whether almost disjoint~$\Gamma$ and~$\Delta$ with elements ${x_1,x_2,y_1,y_2}$ as above exist. The following simple example shows that it is indeed the case. 

\begin{example}
\label{exsecond}
Let ${x_1,y_1\ge 2}$ be coprime integers, and factorize   ${x_1+y_1}$ into coprime factors: ${x_1+y_1=x_2y_2}$ with~$x_2$ and~$y_2$ coprime. Then the four integers $x_1,y_1,x_2,y_2$ are pairwise coprime, which implies that the multiplicative groups 
${\Gamma:=\langle x_1,x_2\rangle}$ and  ${\Delta:= \langle y_1,y_2\rangle}$
are almost disjoint. For instance, ${3+7=2\cdot5}$, and the groups ${\langle 3,2\rangle}$, ${\langle 7,5\rangle}$ are almost disjoint. 
\end{example}

Note also that if ${x_1+y_1,\ x_2+y_2,\  x_3+y_3}$ are multiplicatively dependent, then so are  ${x_1+y_1, \ \xi_2(x_2+y_2),\  \xi_3(x_3+y_3)}$, where~$\xi_2$ and~$\xi_3$ are roots of unity. 

Our second principal result is that  the $abc$-conjecture, or, more precisely, the weak version of this conjecture, see Conjecture~\ref{cowabc} in Subsection~\ref{ssabc},  implies that, with finitely many exceptions,~\eqref{efamthreeone} and~\eqref{ethreesumsminusone} are the only examples of three multiplicatively dependent sums, up to permutation and up to multiplying by roots of unity as above.  

\begin{theorem}
\label{thmthree}
Let~$K$ be a number field containing both~$\Gamma$ and~$\Delta$. Assume that  Conjecture~\ref{cowabc} holds for the field~$K$. 
Let 
${x_1,x_2,x_3\in \Gamma}$ and ${y_1,y_2,y_3\in \Delta}$ be such that the sums 
$
{x_1+y_1, \ x_2+y_2, \ x_3+y_3}
$
are non-zero and multiplicatively dependent, but any two of them are mutiplicatively independent. Then, with finitely many exceptions, after permuting the pairs $(x_k,y_k)$, one of the following holds.
\begin{enumerate}
\item
There exist ${x\in \Gamma}$ and ${y\in \Delta}$  such that
\begin{equation}
\label{emthree}
(x_1,x_2,x_3)= (x,\   \xi_2 x,\ \xi_3x^2), \qquad (y_1,y_2,y_3)= (-y, \ \xi_2y,\ -  \xi_3y^2) 
\end{equation}
with some  roots of unity ${\xi_2,\xi_3\in \Gamma\cap \Delta}$. 
\item
\label{ioptiontwo}
The numbers ${x_1+y_1,\ x_2y_2}$  are multiplicatively dependent  and 
\begin{equation}
\label{ethreetwomone}
x_3=\xi x_2^{-1}, \qquad y_3=\xi y_2^{-1},
\end{equation}
where~$\xi$ is a root of unity.

\end{enumerate}
\end{theorem}

\subsection{A conjecture}

Now we intend to give a conjectural description of~$m$ multiplicatively dependent sums for arbitrary~$m$. To state our conjecture, we need some definitions. Note that these definitions will be used only in this subsection, and not elsewhere in the article. 

A pair  $(\ux,\uy)$ of $m$-dimensional vectors 
 ${\ux:=(x_1, \ldots, x_m)}$ and ${\uy:=(y_1, \ldots, y_m)}$ s with entries in~$\Gamma$, respectively,~$\Delta$, will be briefly called $m$-pair.  We call the $m$-pair $(\ux,\uy)$ \textit{multiplicatively dependent} (MD in the sequel) if the sums 
\begin{equation}
\label{emsums}
x_1+y_1, \ldots , x_m+y_m
\end{equation}
are non-zero and multiplicatively dependent.  We call it \textit{minimally mutiplicatively dependent} (MMD in the sequel) if it is MD and any ${m-1}$ sums among~\eqref{emsums} are multiplicatively independent.

We call two $m$-pairs  $(\ux,\uy)$ and $(\ux',\uy')$  \textit{torsion-equivalent} if 
$$
x_k'=\xi_k x_k, \quad y_k'= \xi_k y_k  \qquad (k=1, \ldots, m), 
$$
where ${\xi_1, \ldots, \xi_m}$ are roots of unity. Clearly, if one of torsion equivalent pairs is MD or MMD, then so is the other. Note also that groups~$\Gamma$ and~$\Delta$ contain only finitely many roots of unity, which implies that each torsion-equivalence class is finite. 

The \textit{concatenation} of two vectors  ${\ux:=(x_1, \ldots, x_m)}$ and ${\ux':=(x_1', \ldots, x_r')}$ is defined by
${\ux\sqcup\ux':=(x_1, \ldots, x_m, x_1', \ldots, x_r')}$. 
The concatenation of pairs is defined entry-wise: 
${(\ux,\uy)\sqcup(\ux',\uy'):= (\ux\sqcup\ux', \uy\sqcup\uy')}$. 
In a similar fashion one defines the concatenation of several vectors or pairs of vectors.

Families~\eqref{efamthreeone}  generalize to arbitrary ${m\ge 2}$. Assume  that~$\Delta$ contains the primitive root of unity $\zeta_{m-1}$.
Let ${x\in \Gamma}$ and ${y\in \Delta}$ be such that ${x^{m-1}\ne y^{m-1}}$.  Then the~$m$ sums 
\begin{equation}
\label{emexample}
x-y, \  x-\zeta_{m-1}y , \ \ldots ,\  x-\zeta_{m-1}^{m-2}y , \  x^{m-1}- y^{m-1} 
\end{equation}
are multiplicatively dependent.  We call $(\ux,\uy)$, where 
$$
\ux:=(\underbrace{x, \ldots, x}_{m-1}, x^{m-1}), \qquad \uy:=-(y,\ \zeta_{m-1}y,\ \ldots,\ \zeta_{m-1}^{m-2}y,\ y^{m-1}), 
$$
the \textit{canonical $m$-pair generated by $(x,y)$}.  Clearly, it is MD. We do not claim that it is always MMD, but it is, with finitely many exceptions (for a fixed~$m$) if Conjecture~\ref{comain} below is true. 
Note that for ${m=2}$, the canonical $2$-pair generated by $(x,-y)$ is  is ${(x,x), (y,y)}$. 

Now set 
\begin{align*}
\ux&:=(x,x^{\eps_2},  \ldots, x^{\eps_{m-1}}, x^{(m-1)\eps_m}),\\
\uy&:=-(y,\ (\zeta_{m-1}y)^{\eps_2},\ \ldots,\ (\zeta_{m-1}^{m-2}y)^{\eps_{m-1}},\ y^{(m-1)\eps_m}), 
\end{align*}
where ${\eps_2, \ldots, \eps_m \in \{1,-1\}}$, and at least one of~$\eps_k$ is~$-1$. Then we call $(\ux,\uy)$ a \textit{twisted $m$-pair generated by $(x,y)$}. A twisted pair is not MD\footnote{To be precise, it can accidentally be MD, but usually it is not.}, but its sums~\eqref{emsums} have a multiplicative combination equal to a power of $xy$. 

Before stating the conjecture, let us give some examples of $4$ multiplicatively dependent sums: 
\begin{alignat*}2
& x-y,\ x-\zeta_3 y,\ x-\zeta_3^2 y,\ x^3-y^3;\\
&x_1+y_1,\ x_2-y_2,\ x_2^{-1}+y_2^{-1},\ x_2^2-y_2^2\qquad  && \text{($x_1+y_1$ and $x_2y_2$ are} \nonumber\\ 
&\qquad && \text{multiplicatively dependent)};\\
&x_1+y_1,\ x_2+y_2, \ x_3+y_3, \ x_3^{-1}+y_3^{-1} \qquad  && \text{($x_1+y_1$, $x_2+y_2$ and $x_3y_3$ are} \nonumber\\ 
&\qquad && \text{multiplicatively dependent)};\\
&x_1+y_1, \ x_1^{-1}+y_1^{-1}, \ x_2+y_2,\ x_2^{-1}+y_2^{-1} \qquad && \text{($x_1y_1$ and $x_2y_2$ are} \nonumber\\ 
&\qquad && \text{multiplicatively dependent)}. 
\end{alignat*}
These examples motivate the following conjecture. 

\begin{conjecture}
\label{comain}
Let ${m\ge 2}$, and let $(\ux,\uy)$ be a MMD $m$-pair. Then, up to permuting the entries and torsion-equivalence, one of the following holds with finitely many exceptions.  

\begin{enumerate}
\item $(\ux,\uy)$ is an $m$-canonical pair.  

\item
\label{icomplic}
There exist 
\begin{itemize}
\item
integers~$r$ and~$s$ satisfying ${r\ge 0}$, ${s\ge 1}$, ${r+s\ge 2}$,  

\item
integers ${m_1, \ldots, m_s\ge2}$ satisfying ${r+m_1+\ldots+m_s=m}$, 

\item
${w_1, \ldots, w_r, t_1, \ldots, t_s \in \Gamma}$, ${z_1, \ldots, z_r, u_1, \ldots, u_s \in \Delta}$ such that the algebraic numbers
$$
w_1+z_1,\ \ldots\, w_r+z_r,\ t_1u_1, \ \ldots, \ t_su_s
$$
are non-zero and multiplicatevely dependent,    
\end{itemize}
and we have 
$$
(\ux,\uy)=(\uw,\uz)\sqcup (\ut^{(1)}, \uu^{(1)})\sqcup\cdots\sqcup (\ut^{(s)}, \uu^{(s)}),
$$ 
where $(\ut^{(k)}, \uu^{(k)})$ is a twisted $m_k$-pair generated by $(t_k, u_k)$. 

\end{enumerate}
\end{conjecture}

Informally, an MMD pair is either canonical, or it has a non-empty part  made of several twisted pairs satisfying a suitable multiplicative dependence condition. 

The case ${m=2}$ of this conjecture is confirmed by Theorem~\ref{thmtwo}, and the case ${m=3}$  is confirmed by Theorem~\ref{thmthree}, assuming $abc$.

\begin{remark}
Generalizing Example~\ref{exsecond}, it is not hard to show that, given~$r$,~$s$ and ${m_1, \ldots, m_s}$ as above, there exist infinitely many pairs of almost disjoint groups~$\Gamma$ and~$\Delta$ of bounded rank, for which there exist an MMD pair as described in item~\ref{icomplic} of the conjecture.  Moreover, if ${\ell_1, \ldots, \ell_s}$ are non-zero integers, then, replacing every $(t_k, u_k)$ by $(t_k^{\ell_k}, u_k^{\ell_k})$, we obtain infinitely many MD-pairs for the same~$\Gamma$ and~$\Delta$. (We do not claim that these pairs are MMD for every choice of ${\ell_1, \ldots, \ell_s}$, but they are for infinitely many choices.)

These observations imply that one cannot expect a
substantially stronger statement than Conjecture~\ref{comain}. 
\end{remark}

\subsection*{Plan of the paper}

In Section~\ref{sheights} we collect necessary facts about heights, recall some results from Baker's theory and state the $abc$-conjecture. In Sections~\ref{spowers} and~\ref{smulti} we obtain auxiliary results used in the proof of Theorem~\ref{thmtwo}. In particular, in Section~\ref{spowers} we show that the sumset of two almost disjoint finitely generated multiplicative groups may contain at most finitely many perfect powers (in a given number field). Theorem~\ref{thmtwo} is proved in Section~\ref{sproofmtwo}. 

Section~\ref{smultdepen} is the technical heart of the article. In it we state and proof Theorem~\ref{thmodgthree}, which classifies (conditionally to $abc$) the triples $(x,y,z)$ of elements of a multiplicative group~$\Lambda$ such that ${x-1,\ y-1,\ z-1}$ are multiplicatively dependent modulo~$\Lambda$. Theorem~\ref{thmthree} is deduced from Theorem~\ref{thmodgthree} in Section~\ref{sproofmthree}. 

\subsubsection*{Acknowledgment} We thank David Masser for a useful discussion. 

\subsubsection*{Disclaimer} No AI tool was used in preparing this article.

\section{Notation and conventions}
\label{sconven}
Letters~$K$ and~$L$ usually denote number fields, letters $\Gamma$,~$\Delta$ and~$\Lambda$ denote finitely generated multiplicative groups of algebraic numbers. Letters $x,y,z,t,u,w$ usually denote algebraic numbers; they are not used in this article as independent variables. 

Letters~$\xi$ and~$\zeta$ are reserved in this article for roots of unity, but they are used differently. If~$\xi$ is supplied with an index, it is just for counting purposes, like $\xi_1,\xi_2,\xi_3$ are just three roots of unity. On the contrary,~$\zeta_m$ denotes a \textit{primitive} root of unity or order~$m$.

Let~$K$ be a number field. 
We denote by~$M_K$ the set of non-trivial absolute values on~$K$,  normalized  to extend standard $p$-adic or infinite absolute values on~$\Q$. That is: if $v\mid \infty$ then ${|x|_v=x}$ for any positive ${x\in \Q}$; and if ${v\mid p}$, where~$p$ is a finite rational prime, then ${|p|_v=p^{-1}}$.  
With this normalization, the product formula on~$K$ looks as
\begin{equation*}
\sum_{v\in M_K} d_v\log|x|_v=0 \qquad (x \in K^\times), 
\end{equation*}
where  ${d_v:=[K_v:\Q_v]}$ are the local degrees.
In the sequel, we call the elements of~$M_K$ \textit{absolute values} or \textit{primes}, depending on what better fits the context. 

When we speak on a finite subset~$S$ of~$M_K$, we always tacitly assume that~$S$ contains all the infinite primes. We denote by $\OO_S$ the ring of~$S$-integers in~$K$, and by $\OO_S^\times$ the group of $S$-units:
$$
\OO_S:=\{x\in K: |x|_v \le 1 \ \text{for}\ v\notin S\}, \qquad \OO_S^\times:=\{x\in K: |x|_v = 1 \ \text{for}\ v\notin S\}. 
$$ 

The letter~$\OO$ is used also in a different context to denote~$\OO_v$, the local ring of the finite prime ${v\in S}$:
$$
\OO_v:= \{x\in K: |x|_v\le 1\}. 
$$ 
This definition is not very coherent with the previously defined $\OO_S$; in fact, 
$$
\OO_S=\bigcap_{v\in M_K\smallsetminus S}\OO_v, 
$$
and, to be consistent, what we denote by~$\OO_S$ must be $\OO_{M_K\smallsetminus S}$. However, both notations~$\OO_S$ and~$\OO_v$ are traditional, and using them will not cause any confusion.

Let~$G$ be a multiplicatively written abelian group and~$m$ a positive integer. In this article  $G^m$ may denote two distinct groups: $\underbrace{G\times\cdots\times G}_m$ and ${\{x^m: x\in G\}}$. It is usually clear from the context in which meaning $G^m$ is used; in the few cases when it may be not clear, it is indicated.

We use the $O(\cdot)$ notation, and we systematically use the Vinogradov notation~$\ll$,~$\gg$ as a substitute for $O(\cdot)$; that is, ${A\ll B}$ means that ${A=O(B)}$ and ${A\gg B}$ means that ${B=O(A)}$. Usually we specify what the implied constants depend of, and whether or not they are effective,  unless it is clear from the context.

\section {Heights, Baker's inequality, $abc$-conjecture}
\label{sheights}

By the \textit{height} of an algebraic number~$x$ we mean the absolute logarithmic height. Recall that it is defined by 
\begin{equation}
\label{edefhe}
\height(x) := d^{-1}\sum_{v \in M_K}d_v\log^+|x|_v , 
\end{equation}
where~$K$ is a number field containing~$x$ and ${d:=[K:\Q]}$. Here and below we are using  the notation 
$$
\log^+:= \max\{\log, 0\}, \qquad \log^\ast := \max\{\log, 1\}.
$$
It is known that the right-hand side of~\eqref{edefhe} depends only on~$x$, not on~$K$. 

We will be using without special reference the standard properties of the height, like 
$$
\height(x+y) \le \height(x)+\height(y) +\log2, \quad \height(xy) \le \height(x)+\height(y), \quad \height(x^n) = |n|\height(x), 
$$
etc. Recall also \textit{Liouville's inequality}: for ${x \in K^\times}$ and ${v\in M_K}$ we have 
\begin{equation}
\label{eliouv}
e^{-(d/d_v)\height(x)}\le |x|_v \le e^{(d/d_v)\height(x)}. 
\end{equation}

We will be also using the following property.

\begin{proposition}
\label{prhpolx}
Let ${F(X) \in \bar\Q[X]}$ be a non-zero polynomial of degree~$m$. Then for any algebraic number~$x$ we have 
\begin{equation*}
\height(F(x)) =m \height(x) + O(1), 
\end{equation*}
where the implied constant depends on~$F$. 
\end{proposition}

More generally, for the vector ${(x_1, \ldots x_m) \in \bar\Q^m}$ we define 
\begin{equation}
\label{edefhem}
\height(x_1, \ldots x_m) := d^{-1}\sum_{v \in M_K}d_v\log^+\max\{|x_1|_v, \ldots, |x_m|_v\}, 
\end{equation}
where~$K$ is a number field containing ${x_1, \ldots, x_m}$ and ${d:=[K:\Q]}$. This is sometimes called \textit{affine height} (to distinguish it from the \textit{projective height}), but we call it simply \textit{height} of ${(x_1, \ldots x_m)}$, because the projective height is not used in this article.

\subsection{Heights on multiplicative groups}

We denote by $\|\cdot\|$ the sup-norm on $\Z^m$. Let  ${\ugamma:=(\gamma_1, \ldots , \gamma_m)\in (\bar\Q^\times)^m}$ and ${\ub =(b_1, \ldots , b_m)\in \Z^m}$. Then the height of ${\ugamma^\ub:= \gamma_1^{b_1}\cdots \gamma_m^{b_m}}$ trivially satisfies 
\begin{equation}
\label{ehmutriv}
\height(\ugamma^\ub) \ll \|\ub\|,
\end{equation}
where the implied constant depends on $\ugamma$. The following observation goes back to Dirichlet.

\begin{proposition}
\label{prdir}
Assume that $\gamma_1, \ldots, \gamma_m$ are multiplicatively independent. Then for any ${\ub \in \Z^m}$ we have  
${\height(\ugamma^\ub) \gg \|\ub\|}$,
where the implied constant depends on $\ugamma$.
\end{proposition}

The proof can be found, for instance,  in \cite[page~77]{Bi02} (immediately after the statement of Theorem~$1'$) or in  \cite[Theorem 13.4]{BBM14}. 

The following consequence is crucial. We know that ${\height(xy)\le \height(x)+\height(y)}$ for any algebraic numbers~$x$ and~$y$. It turns out that, when~$x$ and~$y$ are selected in almost disjoint finitely generated multiplicative groups, an opposite inequality holds. 

\begin{corollary}
\label{coadis}
Let~$\Gamma$ and~$\Delta$ be  almost disjoint finitely generated  subgroups of~$\bar\Q^\times$.  Then for any ${x \in \Gamma}$ and ${y\in \Delta}$ we have ${\height(xy) \gg \height(x)+\height(y)}$, where the implied constant depends on~$\Gamma$ and~$\Delta$. 
\end{corollary}

\begin{proof}

Let~$\Gamma'$ and~$\Delta'$ be  maximal torsion-free subgroups of~$\Gamma$ and~$\Delta$, respectively. Let  ${\gamma_1, \ldots \gamma_r}$ and ${\delta_1, \ldots, \delta_s}$ be free generators of~$\Gamma'$ and~$\Delta'$. Since~$\Gamma$ and~$\Delta$ are  almost disjoint, the ${r+s}$ algebraic numbers ${\gamma_1, \ldots \gamma_r, \delta_1, \ldots, \delta_s}$ are multiplicatively independent. Multiplying~$x$ and~$y$ by some roots of unity, we may assume that ${x\in \Gamma'}$ and ${y\in \Delta'}$. Write ${x=\ugamma^\ub}$  and ${y=\udelta^\uc}$ with ${\ub\in \Z^r}$ and ${\uc\in \Z^s}$. Then 
$$
\height(xy) \gg \max\{\|\ub\|, \|\uc\|\} \gg \max \{\height(x), \height(y)\} \gg \height(x)+\height(y), 
$$
where the first inequality is by Proposition~\ref{prdir} and the second is by~\eqref{ehmutriv}. 
\end{proof}

\subsection{Baker's inequality}
\label{ssbaker}
In this subsection~$K$ is a number field of degree~$d$. 
The following fundamental result is known as \textit{Baker's inequality}.  

\begin{theorem}
\label{thbaker}
Let ${\ugamma=(\gamma_1, \ldots, \gamma_m)\in (K^\times)^m}$  and ${\ub \in \Z^m}$.  Assume that $ {\ugamma^\ub \ne1}$. Then for every ${v\in M_K}$ we have 
$$
|\ugamma^\ub-1|_v \ge e^{-C\height'(\gamma_1)\cdots \height'(\gamma_m) \log^\ast\|\ub\|},
$$
where ${\height'(x):=\min\{\height(x), 1/d\}}$ and ${C>0}$ depends effectively on~$d$, ~$m$ and~$v$. 
\end{theorem}

In the case of infinite~$v$ this can be deduced from  Corollary~2.3 of Matveev~\cite{Ma00}, or from the earlier results of Baker-Wüstholz~\cite{BW93} and Waldschmidt~\cite{Wa93}. The case of finite~$v$ follows from a result of Yu, see the displayed equation on  top of page~30 of~\cite{Yu98}. (Yu improved on his result in subsequent articles, but the improvements mainly concern the shape of the expression for~$C$.)   

Combining Theorem~\ref{thbaker} and Proposition~\ref{prdir}, we obtain the following consequence. 

\begin{corollary}
\label{cobaker}
Let~$\Gamma$ be a finitely generated subgroup of~$K^\times$ and ${v \in M_K}$.  Then there exist ${C>0}$, effectively depending of~$\Gamma$ and~$v$, such that for every ${x\in \Gamma}$ we have either ${x=1}$ or ${|x-1|_v \ge e^{-C\log^\ast\height(x)}}$. 
\end{corollary}


\subsection{Local heights and binary unit equations}

\label{sslocalh}

We continue to assume that~$K$ is a number field of degree~$d$. Given ${x\in K^\times}$, we have 
$$
\height(x) = \height(x^{-1}) = d^{-1} \sum_{v\in M_K} d_v \log^+|x^{-1}|_v = \sum_{v\in M_K}\height_v(x), 
$$
where the summands
$$
\height_v(x) := \frac{d_v}{d}\log^+|x^{-1}|_v 
$$
are called \textit{local heights}. In this notation, Corollary~\ref{cobaker} can be restated as follows: 
\begin{equation}
\label{ecobaker}
\text{for ${x\in \Gamma}$ we have either ${x=1}$ or ${\height_v(x-1) \ll \log^\ast \height(x)}$}, 
\end{equation}
where the implied constant effectively depends on~$\Gamma$ and~$v$.

As an application of Corollary~\ref{cobaker}, we deduce Theorem~\ref{thmone}. For further use we will prove a slightly more general statement. 

\begin{theorem}
\label{thunits}
Let~$\Lambda$  be a finitely generated subgroup of~$\bar\Q^\ast$ and ${\beta \in \bar\Q}$. Then equation ${x+y=\beta}$ has at most finitely many solutions   in ${x,y\in \Lambda}$  and their heights are effectively bounded in terms of~$\Lambda$  and~$\beta$. 
\end{theorem}

Since the field~$K$,  generated by~$\Gamma$ and~$\Delta$, contains at most finitely many roots of unity,  Theorem~\ref{thmone} follows from Theorem~\ref{thunits} used with $\Lambda$ generated by~$\Gamma$, $\Delta$ and the roots of unity from~$K$.

Of course, Theorem~\ref{thunits} is well-known, but it will be used in this article on many occasions, so we prefer to include a quick proof.  

\begin{proof}[Proof of Theorem~\ref{thunits}]
Replacing~$\Lambda$ by the group generated by~$\Lambda$ and~$\beta$, we may assume that  ${\beta \in \Lambda}$.  
Replacing~$x$ and~$y$ by $x/\beta$ and $y/\beta$, the problem reduces to bounding the heights of ${x,y\in \Lambda}$ satisfying ${x+y=1}$.

Let~$K$ be a number field containing~$\Lambda$. Then ${\Lambda \le \OO_S^\times}$ for some finite set ${S\subset M_K}$.  By~\eqref{ecobaker}, for ${x,y\in \OO_S^\times}$ satisfying ${x+y=1}$, we have
$$
\height(y) = \sum_{v\in S} \height_v(y)= \sum_{v\in S}\height_v(x-1) \ll \log^\ast\height(x),
$$
where the implied constant effectively depends on~$S$. Similarly, ${\height(x) \ll \log^\ast\height(y)}$. Hence both $\height(x)$ and $\height(y)$ are effectively bounded. 
\end{proof}

Here is a useful consequence, to be used  in Section~\ref{smultdepen}. 

\begin{corollary}
\label{corat}
Let~$\Lambda$ be as in Theorem~\ref{thunits} and 
 ${F(T)\in \bar\Q(T)}$  a non-zero rational function not of the form ${AT^n}$ for ${A\in \bar\Q^\times}$ and ${n\in \Z}$; that is, either the denominator or the numerator of~$F$ has a  root in~$\bar\Q^\times$. Then there exist at most finitely many ${x\in \Lambda}$ such that ${F(x)\in \Lambda}$, and the heights of these~$x$ are effectively bounded in terms of~$\Lambda$ and~$F$. 
\end{corollary}

\begin{proof}
Write ${F(T)=f(T)/g(T)}$, where ${f(T), g(T)\in \bar\Q[T]}$ are coprime polynomials. By the hypothesis, one of~$f$,~$g$ must have a  root in $\bar\Q^\times$; replacing~$F$ by $F^{-1}$, if needed, we may assume that $f$ has a  root ${\alpha\in \bar\Q^\times}$. Then ${g(\alpha) \ne 0}$.  Expanding~$\Lambda$, we may assume that it contains both~$\alpha$ and $g(\alpha)$. 

Let~$K$ be a number field and~$S$ be a finite subset of $M_K$ such that 
$$
\Lambda \le \OO_S^\times, \qquad f(T), g(T) \in \OO_S[T]. 
$$
Since~$\alpha$ and  $g(\alpha)$ belong to~$\Lambda$, they both  are $S$-integers.  Hence there exist polynomials ${a(T), b(T) \in \OO_S[T]}$ such that 
$$
f(T)=(T-\alpha)a(T), \qquad g(T)-g(\alpha) = (T-\alpha) b(T). 
$$
It follows that for any ${x\in \OO_S}$, the number ${x-\alpha}$ divides (in the ring $\OO_S$) both  $f(x)$ and ${g(x)-g(\alpha)}$. 
Now let ${x \in \OO_S^\times}$ be such that ${F(x)\in \OO_S^\times}$. Then ${x-\alpha \mid f(x)}$ and ${f(x)\mid g(x)}$ in the ring $\OO_S$. Since   ${x-\alpha \mid g(x)-g(\alpha)}$, we obtain ${x-\alpha \mid g(\alpha)}$. Since ${g(\alpha) \in \Lambda \le \OO_S^\times }$, we obtain ${x-\alpha \in \OO_S^\times}$, which bounds the height of~$x$ by Theorem~\ref{thunits}. 
\end{proof}

\subsection{The $abc$-conjecture}
\label{ssabc}

Let~$K$ be a number field. For every finite  ${v \in M_K}$ we denote by $\NN v$ the \textit{absolute norm} of the corresponding prime of~$K$; for an infinite~$v$   we set ${\NN v :=1}$. 

Recall our convention (see  Section~\ref{sconven}): if ${v\in M_K}$ is finite and~$p$ is the underlying rational prime then ${|p|_v=p^{-1}}$. It follows that for ${x\in K}$ with ${|x|_v \ne 1}$ we have ${|x|_v \le p^{-1/e_v}}$ or ${|x|_v \ge p^{1/e_v}}$, where $e_v$ is the ramification index of~$v$ over~$\Q$. In terms of the local heights, as defined in Subsection~\ref{sslocalh}, this implies that
\begin{equation}
\label{eimplies}
\height_v(x) >0\ \Rightarrow \ d\height_v(x) \ge \log \NN v   \qquad (x \in K^\times, \ v \in M_K). 
\end{equation}
Define the conductor of ${x\in K^\times}$ as 
\begin{equation}
\label{econdk}
\cond(x) := d^{-1}\sum_{|x|_v<1} \log \NN v,
\end{equation}
the sum being over ${v\in M_K}$ with the property ${|x|_v <1}$.   Then~\eqref{eimplies} implies the lower bound ${\height(x) \ge \cond(x)}$.

In its simplest form, the \textit{$abc$-conjecture} for the number field~$K$ is as follows.

\begin{conjecture}[$abc$ for a number field]
Let~$K$ be a number field and ${\kappa>1}$. Then for ${x\in K}$, distinct from~$0$ and~$1$, we have 
\begin{equation}
\label{econjv}
\height(x) \le \kappa \bigl(\cond(x)+\cond(x^{-1})+\cond(x-1)\bigr) +O(1), 
\end{equation}
where the bounded term $O(1)$ depends on~$\kappa$ and~$K$. 
\end{conjecture}

In this form the conjecture was originally stated by Vojta \cite[page~84]{Vo87}\footnote{though his notation is quite different from ours}, who extended to number fields the original $abc$-conjecture of Masser and Oesterlé.  To be precise, in Vojta's statement the left-hand side of~\eqref{econjv} is  ${\height(x, x-1)}$ (using our notation). However, since ${\height(x) \le\height(x, x-1) \le \height(x) + \log 2}$, it is equivalent to~\eqref{econjv}. See also \cite{El91,GS00}.


Masser~\cite{Ma02}  suggested a refined version of the $abc$-conjecture with a slightly different definition of the conductor, and this version is more commonly used since then. However, it differs from Vojta's version only in the shape of the bounded term $O(1)$, which is not relevant for our purposes.  Therefore, we stick with the ``old-fashioned'' conductor defined in~\eqref{econdk}, because it is more convenient for us. 

In this article, we will be using the following weak version of the $abc$-conjecture.

\begin{conjecture}[weak $abc$ for a number field]
\label{cowabc}
Let~$K$ be a number field. Then there exists ${\kappa>1}$ such that for ${x\in K}$, distinct from~$0$ and~$1$, we have 
\begin{equation*}
\height(x) \le  \kappa\bigl(\cond(x)+\cond(x^{-1})+\cond(x-1)+1\bigr).  
\end{equation*}
\end{conjecture}

Thus, in the strong version, inequality~\eqref{econjv} must hold for \textit{every} ${\kappa>1}$, while in the weak version it holds for \textit{some}~$\kappa$.

\section{Perfect powers in a sumset}
\label{spowers}

In this section~$K$ is a number field and~$\Gamma$,~$\Delta$ are  finitely generated almost disjoint  subgroups of~$K^\times$.  In this section~$n$ always denotes a positive integer and ${(K^\times)^n:=\{x^n:x\in K^\times\}}$. 
An element of~$K^\times$ is called \textit{perfect power} if it belongs to $(K^\times)^n$ for some  ${n\ge 2}$. 

\begin{theorem}
\label{thpp}
The set of perfect powers in ${\Gamma+\Delta}$ is finite. Moreover, if ${x\in \Gamma}$ and ${y\in \Delta}$ are such that ${x+y}$ is a perfect power, then  the heights of~$x$ and~$y$ are effectively bounded in terms of~$\Gamma$,~$\Delta$ and~$K$. 
\end{theorem}

This theorem is inspired by Shorey and Stewart~\cite{SS83,SS87} and also by the earlier work of Schinzel and Tijdeman~\cite{ST76}. 

The hypothesis of almost disjointness cannot be dropped, see Example~\ref{expothree} in the introduction. 

Theorem~\ref{thpp} is an immediate consequence of the following theorem.

\begin{theorem}
\label{thppz}
Let~$\Lambda$ be a finitely generated subgroup of $K^\times$. Then the set of ${z\in \Lambda }$, satisfying 
\begin{equation}
\label{ezmone}
z+1 \in (K^\times)^n\Lambda
\end{equation} 
with some ${n\ge 2}$, is finite, and the heights of such~$z$ are effectively bounded in terms of~$K$ and~$\Lambda$.  
\end{theorem}

If~$x$ and~$y$ are as in Theorem~\ref{thpp} then ${z:= x/y}$ belongs to ${\Lambda:=\langle\Gamma,\Delta\rangle}$ and satisfies~\eqref{ezmone} with some  ${n\ge 2}$. Theorem~\ref{thppz} bounds the height of~$z$, which bounds the heights of~$x$ and~$y$ by  Corollary~\ref{coadis}. 

\bigskip

We denote ${d:=[K:\Q ]}$.  
Theorem~\ref{thppz} is a consequence of two statements: Proposition~\ref{prschtijd} and Theorem~\ref{thsuper}.

\begin{proposition}
\label{prschtijd}
If ${n\ge1}$ and  ${z\in \OO_S}$ are such that~\eqref{ezmone} holds then either ${\height(z) \le C}$ or ${n\le C}$, where ${C>0}$ effectively depends on~$\Lambda$ and~$d$. 
\end{proposition}

\begin{proof}
The constants in this proof implied by the $\ll$ or $\gg$ notation 
effectively depend on~$\Lambda$ and~$d$.  

Fix a set of generators ${\eta_1, \ldots, \eta_r}$ of~$\Lambda$. Then there exists ${w\in K^\times}$ such that 
$$
z+1=w^n \eta_1^{b_1}\cdots \eta_r^{b_r}, 
$$
where ${b_1, \ldots, b_r \in \Z}$ satisfy
${0\le b_1, \ldots, b_r \le n-1}$. 
Assume first that 
\begin{equation}
\label{ewsmall}
\height(w) \le 2\bigl(\height(\eta_1)+\cdots+\height(\eta_r)\bigr)+d^{-1}. 
\end{equation}
Let ${\Lambda':=\langle w,\Lambda\rangle}$ be the multiplicative group generated by~$w$ and~$\Lambda$. Then ${z+1\in \Lambda'}$, and Theorem~\ref{thunits} bounds $\height(z)$ in terms of~$\Lambda'$, that is, in terms of~$\Lambda$ and $\height(w)$. But since $\height(w)$ is itself bounded in terms of~$\Lambda$, we bound $\height(z)$ in terms of~$\Lambda$. This completes the proof in case~\eqref{ewsmall}.

From now on, we assume that 
\begin{equation}
\label{ewbig}
\height(w) \ge 2\bigl(\height(\eta_1)+\cdots+\height(\eta_r)\bigr)+d^{-1}. 
\end{equation}
This implies that
$$
b_1\height(\eta_1)+\cdots+b_r\height(\eta_r)\le \frac12n\height(w). 
$$
Hence 
\begin{equation}
\label{ehzbig}
\height(z) \ge n \height(w)- b_1\height(\eta_1)-\cdots-b_r\height(\eta_r) \ge \frac12n\height(w). 
\end{equation}
Theorem~\ref{thbaker} implies that, for every ${v \in M_K}$, 
$$
|z|_v \ge \left|w^n \eta_1^{b_1}\cdots \eta_r^{b_r}-1\right|_v \ge e^{-C_1\height(w)\log^\ast  n},
$$
where~$C_1$ effectively depends on~$\Lambda$ and~$d$. (We may write $\height(w)$ instead of $\height'(w)$ due to the term $d^{-1}$ in~\eqref{ewbig}). In terms of the local heights, this can be expressed as 
$$
\height_v(z) \ll \height(w)\log^\ast  n \qquad (v \in M_K). 
$$
Since~$\Lambda$ is a  finitely generated group, there exists a finite set ${S\subset M_K}$ such that ${\Lambda \le \OO_S^\times}$.  It follows that 
$$
\height(z) =\sum_{v\in S} \height_v(z) \ll \height(w)\log^\ast n. 
$$ 
Together with~\eqref{ehzbig}, this bounds~$n$. 
\end{proof}

This argument is, essentially, due to Schinzel and Tijdeman~\cite{ST76}.

\bigskip

The other ingredient needed for the proof of Theorem~\ref{thppz}   is a classical result about superelliptic equations; see, for instance,~\cite{BBGMOS25} and the references therein. 
\begin{theorem}
\label{thsuper}
Let ${f(X) \in K[X]}$ be a polynomial with at least~$3$ simple roots, and ${n\ge 2}$. Let~$S$ be a finite subset of~$M_K$. Then solutions of the equation ${Y^n=F(X)}$ in ${X\in \OO_S}$ and ${Y\in K}$ have heights effectively bounded in terms of~$f$,~$K$ and~$n$. 
\end{theorem}

Now we are ready to complete the proof of Theorem~\ref{thppz}. 

\begin{proof}[Proof of Theorem~\ref{thppz}]
Let ${z\in \Lambda}$  be such that ${z+1\in (K^\times)^n\Lambda}$ with ${n\ge 2}$. Pro\-po\-sition~\ref{prschtijd} implies that one of $\height(z)$ or~$n$ is bounded. If $\height(z)$ is bounded then we are done, so we may assume that~$n$ is bounded.

Let $\Sigma_3$ be a set of representatives of all cosets of $\Lambda/\Lambda^3$. We may select these representatives to have their heights bounded in terms of~$\Lambda$. Then ${z=AX^3}$ with ${A\in \Sigma_3}$ and ${X \in \Lambda\subset \OO_S}$,  where~$S$ is as in the proof of Proposition~\ref{prschtijd}. 
Similarly, select a set~$\Sigma_n$ of bounded\footnote{They are bounded in terms of~$\Lambda$ and~$n$, but, since~$n$ is itself bounded, they are bounded in terms of~$\Lambda$ and~$d$.} representatives of $\Lambda/\Lambda^n$. Then ${z+1= BY^n}$, where ${B\in \Sigma_n}$
and ${Y\in K}$.

Applying Theorem~\ref{thsuper} to the equation ${AX^3+1= BY^n}$, we bound the heights of~$X$ and~$Y$, which bounds the height of~$z$. 
Theorem~\ref{thpp}  is proved.  
\end{proof}

Here is a corollary similar to Corollary~\ref{corat} of Theorem~\ref{thunits}.

\begin{corollary}
\label{coratpow}
Let~$K$ and~$\Lambda$ be as in Theorem~\ref{thppz}. Let  ${F(T)\in K(T)}$  be a non-zero rational function having a simple pole or a  simple zero in~$\bar\Q^\times$;  that is,  the denominator or the numerator of~$F$  has a  simple root in~$\bar\Q^\times$.
Then the set of ${z\in \Lambda }$, satisfying ${F(z) \in (K^\times)^n\Lambda}$ 
with some ${n\ge 2}$, is finite, and the heights of such~$z$ are effectively bounded in terms of~$K$,~$\Lambda$ and~$F$.  
\end{corollary}

\begin{proof}
Write ${F(T)=f(T)/g(T)}$, where ${f(T), g(T)\in K[T]}$ are coprime polynomials.  By the hypothesis, one of~$f$,~$g$ must have a simple root in $\bar\Q^\times$ and we may assume that~$f$ has such a root, denoted by~$\alpha$. Expanding the field~$K$ and the group~$\Lambda$, we may assume that~$\Lambda$ contains~$\alpha$, $f'(\alpha)$  and $g(\alpha)$. 

Let~$S$ be a finite subset of $M_K$ such that 
$$
\Lambda \le \OO_S^\times, \qquad f(T), g(T) \in \OO_S[T],  
$$
and~$\OO_S$ is a principal ideal ring. Set ${e(T):=f(T)/(T-\alpha)}$. Then ${e(T)\in \OO_S[T]}$ and ${e(\alpha) =f'(\alpha)\in \OO_S^\times}$. There exist polynomials ${a(T), b(T) \in \OO_S|T]}$ such that 
$$
e(T)-a(T)(T-\alpha)=e(\alpha) \in \OO_S^\times, \qquad g(T)-b(T) (T-\alpha)=g(\alpha) \in \OO_S^\times. 
$$
It follows that for any ${z\in \OO_S}$, the number ${z-\alpha}$ is coprime with $e(z)$ and with $g(z)$ in the principal ideal ring $\OO_S$. 

Now let ${z\in \OO_S^\times}$ be such that 
$$
F(z) = \frac{(z-\alpha)e(z)}{g(z)} \in (K^\times)^n\OO_S^\times 
$$
with ${n\ge 2}$. By the coprimarity indicated above, the principal ideal ${(z-\alpha)}$ must be an $n$th power of another ideal of $\OO_S$. Since it is a principal ideal ring, we have ${z-\alpha \in (K^\times)^n\OO_S^\times}$, and Theorem~\ref{thppz} bounds $\height(z)$.  
\end{proof}


\section{Equations in multiplicative groups}
\label{smulti}

Let~$\Gamma$ and~$\Delta$ be finitely generated almost disjoint subgroups of $\bar\Q^\times$. In this section we study the equation ${x_1+y_1=x_2+y_2}$ in ${x_1,x_2\in \Gamma}$ and ${y_1,y_2\in \Delta}$. We use the classical theory of exponential  Diophantine equations to show, non-effectively, that this equation may have at most finitely many non-trivial solution (those with ${x_1\ne x_2}$), see Proposition~\ref{proo}. We also show (see Theorem~\ref{thrankone}) that this can be made effective in the important special case when both~$\Gamma$ and~$\Delta$ are of rank~$1$ (that is, a product of a finite cyclic group and an infinite cyclic group).

We will be using the  classical theorem about exponential Diophantine equations, proved independently by Evertse,  Schlickewei and van der Poorten, see \cite[Theorem~1]{Ev84}, \cite[Théorème~1]{La84} and the references therein.

Let~$\Lambda$ be a finitely generated subgroup of~$\bar\Q^\times$.   We are interested in solutions ${\ux=(x_1,\ldots,x_m)\in \Lambda^m}$ of the equation 
\begin{equation}
\label{eexp}
x_1+\cdots+x_m=0. 
\end{equation}
We call two solutions ${\ux, \uy \in \Lambda^m}$ \textbf{equivalent} if ${\uy=\lambda\ux}$ for some ${\lambda\in\Lambda}$. We call a solution~$\ux$ \textit{primitive} if no proper sub-sum of ${x_1+\cdots+x_m}$ vanishes.

\begin{theorem}[Evertse, Schlickewej, van der Poorten]
\label{thexp}
Equation~\eqref{eexp} has, up to equivalence, at most finitely many primitive solutions ${\ux\in \Lambda^m}$. 
\end{theorem}

Theorem~\ref{thexp}  is non-effective (its proof is based on the subspace theorem).


\subsection{A non-effective result}

In this subsection we use Theorem~\ref{thexp} to prove the following statement. 

\begin{proposition}
\label{proo}
Let~$\Gamma$ and~$\Delta$ be  finitely generated almost disjoint subgroups of $\bar\Q^\times$. 
Then the equation 
\begin{equation}
\label{erone}
x_1+y_1=x_2+y_2
\end{equation}
has at most finitely many solutions satisfying
\begin{equation}
\label{econds}
x_1,x_2\in \Gamma, \qquad y_1,y_2\in \Delta, \qquad x_1\ne x_2. 
\end{equation}
\end{proposition}

\begin{proof}[Proof of Proposition~\ref{proo}]
Let~$\Lambda$ be the group generated by~$\Gamma$ and~$\Delta$. By Theorem~\ref{thexp}, the equation ${z_1+z_2+z_3+z_4=0}$ has, up to equivalence, at most  finitely many primitive solutions ${\uz\in \Lambda^4}$ . Since~$\Gamma$ and~$\Delta$ are almost disjoint,  each equivalence class may contain at most finitely many solutions of the form $(x_1,y_1,-x_2, -y_2)$ with ${x_1,x_2 \in \Gamma}$ and ${y_1,y_2\in \Delta}$. 

Now assume that $(x_1,y_1,- x_2, - y_2)$ is not a primitive solution.  Then we have one of the three options:
\begin{align}
\label{exoyo}
&x_1+y_1=0, \qquad x_2+y_2=0;\\
\label{exoyt}
&x_1-y_2=0, \qquad y_1- x_2=0;\\
\label{exoxt}
&x_1-x_2=0, \qquad y_1-y_2=0. 
\end{align}
Since~$\Gamma$ and~$\Delta$ are almost disjoint, options~\eqref{exoyo} and~\eqref{exoyt} may occur only finitely often. We have proved that, with finitely many exceptions, we have~\eqref{exoxt}. 
\end{proof}

Since Theorem~\ref{thexp}  is non-effective, so is Proposition~\ref{proo}. 

\subsection{An effective result}
Proposition~\ref{proo} can be made effective in the (quite important) special  case when both~$\Gamma$ and~$\Delta$ are of rank~$1$.

\begin{theorem}
\label{thrankone}
In the set-up of Proposition~\ref{proo}, assume in addition that~$\Gamma$ and~$\Delta$ are both of rank~$1$. Then  the solutions of~\eqref{erone}, satisfying~\eqref{econds}, have heights effectively bounded in terms of~$\Gamma$ and~$\Delta$. 
\end{theorem}

This theorem will be proved in Subsection~\ref{ssproofrankone}  using the following  simple lemma.

\begin{lemma}
\label{lerhow}
Let~$K$ be a number field and ${\alpha, \beta\in K^\times}$. Assume that for every ${v\in M_K}$ we have 
\begin{equation}
\label{ebothneone}
|\alpha|_v \ne 1 \Leftrightarrow |\beta|_v\ne 1. 
\end{equation}
Assume further that all the quotients 
\begin{equation}
\label{equotients}
\frac{\log |\beta|_v}{\log|\alpha|_v} \qquad (|\alpha|_v \ne 1)
\end{equation}
are equal. Then~$\alpha$ and~$\beta$ are multiplicatively dependent. 
\end{lemma}

\begin{proof}
Let $S=\{v_1, \ldots, v_s\}$ be the subset of~$M_K$ consisting of~$v$ with ${|\alpha|_v \ne 1}$.  Then~$\alpha$ is an $S$-unit, and so is~$\beta$, by~\eqref{ebothneone}. 
Let ${\eta_1, \ldots, \eta_{s-1}}$ be a basis of a maximal infinite subgroup of the group of~$S$-units. Then 
$$
\alpha=\xi\eta_1^{m_1}\cdots \eta_{s_1}^{m_{s-1}}, \qquad \alpha=\theta\eta_1^{n_1}\cdots \eta_{s_1}^{n_{s-1}}, 
$$
where ${m_1, \ldots, m_{s-1}, n_1, \ldots, n_{s-1} \in \Z}$ and ${\xi, \theta}$ are roots of unity. We denote by~$\um$
and~$\un$ the vertical vectors
${[m_1, \ldots, m_{s-1}] }$  and ${[n_1, \ldots, n_{s-1}] }$, respectively. To prove the lemma, we have to show that these vectors are colinear: there exists ${\lambda \ne 0}$ such that ${\un=\lambda \um}$. 

Denote by~$E$ the matrix ${[\log|\eta_j|_{v_i}]_{1\le i,j\le s-1}}$.  Equality of the quotients~\eqref{equotients} means that the vectors ${E\um}$ and ${E\un}$ are colinear. Since~$E$ is a nonsingular matrix, this implies that~$\um$
and~$\un$ are colinear as well. The lemma is proved. 
\end{proof}

\subsection{Proof of Theorem~\ref{thrankone}}
\label{ssproofrankone}

To start with, let us note that it suffices to bound the heights of any two of the four variables $x_1,x_2,y_1,y_2$. Indeed, if, say, $\height(x_1)$ and $\height(y_2)$ are bounded, then, applying Theorem~\ref{thunits} to the equation ${x_2-y_1=y_2-x_1}$, we bound the heights of~$x_2$ and~$y_1$ as well. 

Let~$\gamma$ and~$\delta$ generate the maximal infinite subgroups of~$\Gamma$ and~$\Delta$, respectively. We will use notation 
\begin{equation}
\label{erhov}
\rho_v :=\frac{\log|\delta|_v}{\log|\gamma|_v} \qquad (v \in M_K, \quad |\gamma|_v\ne 1). 
\end{equation}

The proof is rather long, so we split it into several steps.

\subsubsection{Exponents}
We have 
$$
x_i=\xi_i\gamma^{m_i}, \quad y_i =\theta_i\delta^{n_i} \qquad (i=1,2), 
$$  
where ${m_1,m_2,n_1,n_2\in \Z}$ and ${\xi_1,\xi_2,\theta_1, \theta_2}$ are roots of unity.  In the sequel,  the integers ${m_1,m_2,n_1,n_2}$
will be called \textit{exponents}. Note that 
$$
\height (x_i)= |m_i|\height(\gamma), \qquad \height (y_i)= |n_i|\height(\delta). 
$$
Hence, to prove the theorem, we have to bound the absolute values of any two of the four exponents.

We may assume, replacing, if necessary,~$\gamma$ by~$\gamma^{-1}$, and renumbering the variables,  that ${m_1>0}$ and that ${m_1\ge |m_2|}$. In addition to this, we may assume that ${m_1\ne m_2}$. Indeed, if ${m_1=m_2}$ then ${\xi_1\ne \xi_2}$: otherwise we violate the hypothesis ${x_1\ne x_2}$. We obtain ${x_2=\xi x_1}$ with a root of unity ${\xi\ne 1}$ belonging to the maximal finite subgroup of the group $\Gamma\Delta$. Our equation now becomes ${y_1x_1^{-1}- y_2x_1^{-1}= \xi-1}$. For every given~$\xi$, the solutions of this equation are effectively bounded by Theorem~\ref{thunits}, which bounds the  initial variables~$x_i$ and~$y_i$ by Corollary~\ref{coadis}. 
Similarly, we may assume that ${n_1\ne n_2}$. 

We may also assume that ${m_1\ge 3}$ (this is just for convenience, to be allowed to write $\log m_1$ where $\log^\ast m_1$ is expected). Indeed, otherwise both ${|m_1|=m_1}$ and ${|m_2|\le m_1}$ are bounded and we are done. 

Let us summarize what  we know  so far about the exponents $m_i$ and $n_i$:
\begin{equation}
\label{ewhatweknow}
m_1\ge 3, \quad m_1\ge |m_2|, \quad m_1>m_2, \qquad  n_1\ne n_2. 
\end{equation}

\subsubsection{Study of a fixed~$v$ with ${|\gamma|_v >1}$}
Let~$K$ be a number field containing both groups~$\Gamma$  and~$\Delta$. 
Since~$\gamma$ is not a root of unity, there exists ${v \in M_K}$ such that ${|\gamma|_v >1}$.   Since ${m_1>m_2}$, we have the following: 
$$
|x_1-x_2|_v= |x_1|_v 
$$
if~$v$ is a finite absolute value, and 
$$
(1-|\gamma|_v^{-1}) \le  \frac{|x_1-x_2|_v}{|x_1|_v} <2, 
$$
if~$v$ is infinite.
In any case,  
\begin{equation}
\label{exoxtxo}
\log|x_1-x_2|_v = \log|x_1|_v +O(1), 
\end{equation}
where the implied constant depends on~$\gamma$.  In general, in the rest of the proof the implied constants and the constants denoted by $C,C_1,C_2,\ldots$ effectively depend on~$\gamma$ and~$\delta$.

If ${|\delta|_v =1}$ then ${|y_1-y_2|_v\le 2}$ (and even $\le 1$ if~$v$ is finite). Since 
\begin{equation}
\label{eyetanother}
|y_1-y_2|_v =|x_1-x_2|_v, 
\end{equation}
equation~\eqref{exoxtxo} implies that $ \log|x_1|_v $ is bounded from above.  Hence~$m_1$ is bounded from above. From~\eqref{ewhatweknow}, this bounds ${|m_1|=m_1}$ and ${|m_2|\le m_1}$, hereby completing the proof.  

Now assume that ${|\delta|_v \ne 1}$. Replacing (if needed)~$\delta$ by $\delta^{-1}$ and $n_1,n_2$ by $-n_1, -n_2$, we may assume that ${|\delta|_v>1}$. If both ${n_1,n_2\le 0}$ then ${|y_1-y_2|_v\le 2}$ and we complete the proof as in the previous paragraph. Now assume that one of $n_1,n_2$ is strictly positive. We may rewrite~\eqref{erone}  as 
${x_1+(-y_2)=x_2+(-y_1)}$
Hence, replacing $y_1,y_2$ by $-y_2,-y_1$, and recalling that  ${n_1\ne n_2}$ by~\eqref{ewhatweknow}, we may assume that 
\begin{equation}
n_1>0, \qquad n_1>n_2. 
\end{equation}
Arguing as before, we obtain for $y_1,y_2$ a relation analogous to~\eqref{exoxtxo}:
$$
\log|y_1-y_2|_v = \log|y_1|_v +O(1). 
$$
Together with~\eqref{exoxtxo} and~\eqref{eyetanother}, this implies that 
\begin{equation*}
\log|y_1|_v =\log|x_1|_v+O(1);  
\end{equation*}
or, in terms of the exponents,
\begin{equation}
\label{enonemone}
n_1 =\rho_vm_1+O(1),  
\end{equation}
where~$\rho_v$ is defined in~\eqref{erhov}.

\subsubsection{Using Baker I}
Our next step is to estimate ${|x_1+y_1|_v}$ from below and from above. We have 
$$
\height(x_1^{-1}y_1) \ll m_1+n_1\ll m_1. 
$$
Corollary~\ref{coadis} implies that 
$$
|x_1+y_1|_v =|x_1|_v |-x_1^{-1}y_1-1|_v \ge  |x_1|_v e^{-C_1\log^\ast\height(-x_1^{-1}y_1)} \gg |x_1|_v e^{-C_1\log m_1} , 
$$
and, similarly, ${|x_1+y_1|_v \gg |y_1|_v e^{-C_1\log m_1}}$. 
On the other hand, 
$$
|x_1+y_1|_v =|x_2+y_2|_v \le 2\max\{|x_2|_v, |y_2|_v\}. 
$$
Hence one of the following holds:
\begin{equation*}
\log|x_2|_v \ge \log|x_1|_v -O(\log m_1) \quad \text{or} \quad \log|y_2|_v \ge \log|y_1|_v -O(\log  m_1). 
\end{equation*}
Or, in terms of the exponents, 
\begin{align}
\label{emm}
m_2 &\ge m_1  -O(\log  m_1) \quad \text{or} \\
\label{enn} 
n_2 &\ge n_1  -O(\log  m_1) . 
\end{align}

\subsubsection{Using Baker II}
Assume~\eqref{emm}. We want to show that a weaker version of~\eqref{enn} holds as well. 
We are going to use Baker to estimate from below the quantity 
$$
\left|\frac{y_2}{y_1}\right|_v = e^{(n_2-n_1)\log |\delta|_v}. 
$$
We have
$$
-\frac{y_2}{y_1} = \frac{x_2-x_1-1}{ y_1} =\xi \delta^{-n_1}\gamma^{m_1} z -1,  
$$
where
$$
z:= \frac{x_2}{x_1}-1= \theta \gamma^{m_2-m_1} -1 
$$
and~$\xi$ and~$\theta$  are roots of unity. By~\eqref{emm},  
$$
\height(z)  \ll m_1-m_2 \ll \log  m_1
$$
Since ${n_1\ll m_1}$ by~\eqref{enonemone}, Theorem~\ref{thbaker} implies that 
$$
\left|\frac{y_2}{y_1}\right|_v = |\xi \delta^{-n_1}\gamma^{m_1} z -1|_v \ge e^{-C_2(\log m_1)^2}. 
$$
Hence 
${n_2 \ge n_1 -O(\log m_1)^2}$.

Similarly, if~\eqref{enn} holds, then ${m_2 \ge m_1 -O(\log m_1)^2}$. Thus, in any case we have
$$
m_2 \ge m_1  -O(\log  m_1)^2, \qquad  
n_2 \ge n_1  -O(\log  m_1)^2 . 
$$
Since ${m_1>m_2}$ and ${n_1>n_2}$, we have just proved that all the four exponents are of the same order of magnitude:
\begin{align}
\label{emagnitudem}
m_2 &= m_1 + O(\log m_1)^2, \\ 
\label{emagnituden}
n_1, n_2 &= \rho_v m_1 + O(\log m_1)^2. 
\end{align}
In particular, to complete the proof, it suffices to bound~$m_1$. 

Note also that we may assume that all of the exponents are positive: otherwise  ${m_1=O(\log m_1)^2}$, which bounds~$m_1$. Thus,
\begin{equation}
m_1>m_2>0, \qquad n_1>n_2>0. 
\end{equation} 

\subsubsection{A different absolute value}
We obtained  strong asymptotic formulas~\eqref{emagnitudem} and~\eqref{emagnituden} by studying just a single fixed ${v\in M_K}$ with the property ${|\gamma|_v>1}$ and ${|\delta|_v>1}$. 
Now let  ${w\in M_K}$ be a different absolute value. Assume  that ${|\gamma|_w >1}$. Then, arguing as above, we obtain the analogue of~\eqref{exoxtxo}:
\begin{equation}
\label{exow}
\log|x_1-x_2|_w = \log|x_1|_w +O(1)= m\log|\gamma|_w+O(1).   
\end{equation}
If ${|\delta|_w \le 1}$ then ${|x_1-x_2|_w=|y_1-y_2|_w\le 1}$,  because ${n_1,n_2>0}$; this, together with~\eqref{exow}, bounds~$m_1$. Hence we may assume that ${|\delta|_w>1}$. Similarly,  ${|\delta|_w>1}$ implies that ${|\gamma|_w>1}$. Thus
\begin{equation}
\label{egeone}
|\gamma|_w >1 \Leftrightarrow |\delta|_w>1 \qquad (w \in M_K). 
\end{equation}
Moreover, in the case ${|\gamma|_w>1}$ the same argument as above shows that asymptotic relations~\eqref{emagnituden} hold for~$w$ as well: 
\begin{equation}
\label{emnrhow}
n_1, n_2 = \rho_w m_1 + O(\log m_1)^2 \qquad (w \in M_K, \quad |\gamma|_w>1).  
\end{equation}

Now assume that ${|\gamma|_w <1}$. Then we cannot have ${|\delta|_w>1}$ by~\eqref{egeone}. Assume that ${|\delta|_w=1}$. Since ${m_1>m_2>0}$, we have 
\begin{equation}
\label{enegw}
\log|x_1-x_2|_w =   m_2\log|\gamma|_w+O(1)= m_1\log|\gamma|_w+O(\log m_1)^2.   
\end{equation}
On the other hand, ${|y_1|_w=|y_2|_w=1}$, and we obtain 
\begin{equation}
\label{edelwone}
\log|x_1-x_2|_w = \log|y_1-y_2|_w= \log\left|\frac{y_1}{y_2}-1\right|_w. 
\end{equation}
We have ${\height(y_1/y_2-1) \le (n_1-n_2)\height(\delta) +\log2 \ll (\log m_1)^2}$. 
By Liouville's inequality~\eqref{eliouv}, 
$$
\left|\frac{y_1}{y_2}-1\right|_w \ge e^{-(d/d_w) \height(y_1/y_2)} \ge e^{-C_3(\log m_1)^2}. 
$$
Combining this with~\eqref{enegw} and~\eqref{edelwone} and noting that ${\log|\gamma|_w<0}$, we obtain ${m_1 \ll (\log m_1)^2}$, which bounds~$m_1$. 

Thus, we may may assume that ${|\delta|_w<1}$ when ${|\gamma|_w<1}$. Similarly, ${|\delta|_w<1}$ implies ${|\gamma|_w<1}$. Thus, 
\begin{equation}
\label{eleone}
|\gamma|_w <1 \Leftrightarrow |\delta|_w<1 \qquad (w \in M_K). 
\end{equation}
Furthermore, we have an analogue of~\eqref{enegw}:
\begin{equation}
\log|y_1-y_2|_w = n_1\log|\delta|_w +O(\log m_1)^2.   
\end{equation}
By~\eqref{enegw}, this implies that ${n_1=\rho_w m_1+O(\log m_1)^2}$ when ${|\gamma|_w <1}$. Together with~\eqref{emnrhow}, we obtain  that 
\begin{equation}
\label{eallw}
n_1=\rho_w m_1+O(\log m_1)^2 
\end{equation}
holds for all~$w$ satisfying ${|\gamma|_w \ne 1}$.

\subsubsection{Completing the proof}
Now we are ready to complete the proof. 
Since groups~$\Gamma$ and~$\Delta$ are almost disjoint,~$\gamma$ and~$\delta$ must be algebraically independent. 
Lemma~\ref{lerhow}, together with~\eqref{egeone} and~\eqref{eleone}, implies that the numbers~$\rho_w$ are not all equal: there exist ${w,w'\in  M_K}$ with ${|\gamma|_w, |\gamma|_{w'}\ne 1}$  such that ${\rho_{w}\ne \rho_{w'}}$. 
Denote ${\eps:=|\rho_{w}- \rho_{w'}|}$. Then~\eqref{eallw} and its analog with~$w'$ imply  that 
${\eps m_1 \ll (\log m_1)^2}$. This bounds~$m_1$. Theorem~\ref{thrankone} is proved. \qed

\section{Proof of Theorem~\ref{thmtwo}}

\label{sproofmtwo}

Now we are ready to prove Theorem~\ref{thmtwo}. Let~$K$ be a number field containing~$\Gamma$ and~$\Delta$. 

Since ${x_1+y_1}$ and ${x_2+y_2}$ are multiplicatively dependent, 
we have 
$$
x_1+y_1= \xi_1 z^{n_1}, \qquad x_2+y_2= \xi_2 z^{n_2}, 
$$
where~$n_1$ and~$n_2$ are  non-zero integers,~$\xi_1$ and~$\xi_2$ are roots  of unity and  ${z \in K^\times}$. If~$z$ is a root of unity then we have only finitely possibilities for~$x_k$ and~$y_k$ by Theorem~\ref{thmone}. From now on,~$z$ is not a root of unity.  Without loss of generality ${n_1>0}$ and ${n_1\ge |n_2|}$. 

If ${n_1\ge2}$ then the heights of~$x_1$ and~$y_1$ are bounded by Theorem~\ref{thpp}.  Hence the height of ${z^{n_1}=(x_1+y_1)\xi_1^{-1}}$ is bounded. Since ${|n_2|\le n_1}$, this also bounds the height of $z^{n_2}$. Now Theorem~\ref{thunits} bounds the heights of~$x_2$ and~$y_2$.

If ${n_1=1}$ then ${n_2=1}$ or~$-1$. When ${n_1=n_2=1}$,  we obtain 
$$
x_1+y_1 =\xi x_2+\xi y_2
$$
for some root of unity ${\xi\in K}$.  We may assume, without loss of generality, that~$\Delta$ contains all the roots of unity from~$K$; in particular, ${\xi \in \Delta}$. Proposition~\ref{proo} implies now that, with finitely many exceptions, ${x_1= \xi x_2}$ and ${y_1=\xi y_2}$. 

Now assume that ${n_1=1}$ and ${n_2=-1}$. Then we have
\begin{equation}
\label{eprod}
(x_1+y_1)(x_2+y_2)=\xi. 
\end{equation}
with some root of unity ${\xi\in K}$. Let~$S$ be a finite subset~$M_K$ such that  $\OO_S^\times$ contains both~$\Gamma$ and~$\Delta$. Then both ${x_1+y_1}$ and ${x_2+y_2}$ are $S$-integers, and~\eqref{eprod} implies that both are $S$-units:
\begin{equation}
\label{eredtounit}
x_i+y_i=:z_i \in \OO_S^\times \qquad (i=1,2). 
\end{equation}
Applying Theorem~\ref{thunits} with ${\Lambda=\OO_S^\times}$ to the equation ${x_iy_i^{-1}-z_iy_i^{-1}=-1}$, we bound the height of $x_iy_i^{-1}$. By Corollary~\ref{coadis}, this bounds the heights of~$x_i$ and~$y_i$. 
This completes the proof of Theorem~\ref{thmtwo}.  \qed

\begin{remark}
\label{reff}
All steps of this proof are effective, except using Proposition~\ref{proo} in the case ${n_1=n_2=1}$. This makes Theorem~\ref{thmtwo} non-effective in general. However, Proposition~\ref{proo} is effective in the special case when both groups~$\Gamma$ and~$\Delta$ are of rank~$1$, see Theorem~\ref{thrankone}. Hence Theorem~\ref{thmtwo} is effective in this case. 
\end{remark}

\section{Multiplicative dependence modulo a group}
\label{smultdepen}

Let~$\Lambda$ be subgroup of $\bar\Q^\times$. We say that ${x_1, \ldots, x_m\in \bar\Q^\times}$ are \textit{multiplicatively dependent modulo~$\Lambda$} if there exist integers ${a_1, \ldots, a_m}$, not all~$0$, such that 
$$
x_1^{a_1}\cdots x_m^{a_m}\in \Lambda. 
$$
In particular, for ${m=1}$, a singleton~$x$ is multiplicatively dependent modulo~$\Lambda$ if ${x\in \sqrt\Lambda}$, where
$$
\sqrt\Lambda:=\{\alpha \in \bar\Q^\times: \alpha^n \in \Lambda \ \text{for some positive integer~$n$}\}
$$  
is usually called the \textit{radical} of~$\Lambda$. 

Now let~$\Lambda$ be a finitely generated group. We are interested in the following question: for a given~$m$, determine 
the $m$-tuples ${(x_1, \ldots, x_m)\in \Lambda^m}$ such that ${x_1-1, \ldots, x_m-1}$ are multiplicatively dependent modulo ~$\Lambda$. The question is interesting on its own, and is closely related to Conjecture~\ref{comain}.  We state the corresponding conjecture in Subsection~\ref{ssconj}.  
However, in  this section we  are mainly concerned the cases ${m\le 3}$. The corresponding results will be used in Section~\ref{sproofmthree} to prove Theorem~\ref{thmthree}.

For ${m=1}$ it is easy.
\begin{theorem}
\label{thmodgone}
Let~$\Lambda$ be a finitely generated subgroup of $\bar\Q^\times$. Then 
there exist at most finitely many ${x\in \Lambda}$ such that ${x-1\in \sqrt\Lambda}$, and the heights of such~$x$ are effectively bounded in terms of~$\Lambda$.
\end{theorem}
\begin{proof}
Let~$K$ be a number field containing~$\Lambda$. If ${x \in \Lambda}$ and ${x-1\in \sqrt\Lambda}$ then ${x-1}$ belongs to ${\sqrt\Lambda\cap K^\times}$, which is a finitely generated group. Hence $\height(x)$ is effectively bounded by Theorem~\ref{thunits}. 
\end{proof}

For ${m=2}$ there are infinitely many examples: for every ${x\in \Lambda}$, distinct from~$1$, the numbers ${x-1\ ,x-1}$ are multiplicatively dependent modulo~$\Lambda$, and so are ${x-1,\  x^{-1}-1}$. In the following theorem we show that these are the only examples, with finitely many exceptions. 

\begin{theorem}
\label{thmodgtwo}
Let~$\Lambda$ be a finitely generated subgroup of $\bar\Q^\times$.  Let ${x,y\in \Lambda}$ be distinct from~$1$  and satisfying
\begin{equation}
\label{enotinrad}
x-1, y-1\notin\sqrt\Lambda. 
\end{equation}
Assume that ${x-1}$ and ${y-1}$ are multiplicatively dependent modulo~$\Lambda$. 
Then, with finitely many exceptions, we have\footnote{Here and elsewhere   ${y= x^{\pm1}}$ means that ${y=x}$ or ${y=x^{-1}}$.}  ${x=y^{\pm1}}$. 
\end{theorem}

For ${m=3}$ there is a new infinite series of examples: given ${x\in \Lambda}$  distinct from~$1$, set
${y:= -x^{\pm1}}$  and ${z:= x^{\pm2}}$; then ${x-1, \ y-1, \ z-1}$ are multiplicatively dependent modulo~$\Lambda$. 
The following theorem shows that, conditionally to the weak $abc$-conjecture, these are the only examples, with finitely many exceptions. 

\begin{theorem}
\label{thmodgthree}
Let~$K$ be a number field and~$\Lambda$ a finitely generated subgroup of $K^\times$. Assume that the weak $abc$-conjecture (Conjecture~\ref{cowabc}) holds for the field~$K$. Let ${x,y,z\in \Lambda}$ be distinct  from~$1$. Assume that ${x-1, \ y-1, \ z-1}$ are multiplicatively dependent modulo~$\Lambda$, but any two of these numbers are multiplicatively independent modulo~$\Lambda$. 
Then, with finitely many exceptions, after  a permutation of $x,y,z$, we have ${y= -x^{\pm1}}$  and ${z= x^{\pm2}}$. 
\end{theorem}

Theorem~\ref{thmodgtwo} is proved in Subsection~\ref{ssmodgtwo}. In Subsections~\ref{ssgcd} and~\ref{ssprim} we develop various tools needed for the proof of Theorem~\ref{thmodgthree}. This latter is proved in Subsection~\ref{ssproofmodg}.

\subsection{Proof of Theorem~\ref{thmodgtwo}}
\label{ssmodgtwo}

In this subsection we prove Theorem~\ref{thmodgtwo}. By the hypothesis, there exist ${a,b\in \Z}$, not both~$0$,   such that 
\begin{equation}
\label{exayb}
(x-1)^a(y-1)^b\in  \Lambda. 
\end{equation}
Moreover, both~$a$ and~$b$ are not~$0$; indeed, if, say, ${b=0}$ then ${x-1\in \sqrt\Lambda}$, contradicting~\eqref{enotinrad}. 

Our first observation is that bounding the height of one of the unknowns $x,y$ bounds also the other. More precisely, we have the following.

\begin{proposition}
\label{prhxhyb}
Let ${A>0}$. In the set-up of Theorem~\ref{thmodgtwo}, assume that ${\height(x) \le A}$. Then $\height(y)$ is bounded in terms of~$A$ and~$\Lambda$. 
\end{proposition}

\begin{proof}
Let~$K$ be the number field generated by~$\Lambda$ and~$\Lambda'$  the multiplicative group, generated by~$\Lambda$ and by all  ${\alpha \in K^\times}$ satisfying ${\height(\alpha)\le A+\log 2}$. By Northcott's theorem, there exist  at most finitely many~$\alpha$ as above, which implies that~$\Lambda'$ is a finitely generated group. 

If ${\height(x) \le A}$ then ${x-1\in \Lambda'}$ by the definition of~$\Lambda'$. Hence ${y-1\in \sqrt{\Lambda'}}$ by~\eqref{exayb}. Theorem~\ref{thmodgone} now  bounds $\height(y)$ in terms of~$\Lambda'$, that is,  in terms of~$A$ and~$\Lambda$. 
\end{proof}

Our second observation is that one may replace~$\Lambda$ by a bigger group.

\begin{proposition}
\label{prreplace}
Let $\Lambda'$ be a finitely generated subgroup of $\bar\Q^\times$ containing~$\Lambda$. 
Assume that Theorem~\ref{thmodgtwo} holds with~$\Lambda'$ instead of~$\Lambda$. Then it holds with~$\Lambda$ as well. 
\end{proposition}

\begin{proof}
We have to show the following: if ${x,y\in \Lambda}$ satisfy~\eqref{enotinrad} and~\eqref{exayb} with ${a,b\ne 0}$,  then, with finitely many exceptions, they satisfy ${x-1, y-1\notin\sqrt{\Lambda'}}$. Thus, assume that ${x-1\in \sqrt{\Lambda'}}$. Then ${y-1\in \sqrt{\Lambda'}}$ as well by~\eqref{exayb}. Now Theorem~\ref{thmodgone} bounds both $\height(x)$ and $\height(y)$, whence the result. 
\end{proof}

Now we are ready to complete the proof of Theorem~\ref{thmodgtwo}.  Let~$K$ be the number field generated by~$\Lambda$ and let~$S$ be a finite subset of $M_K$ such that ${\Lambda\le \OO_S^\times}$. By Proposition~\ref{prreplace}, we may assume that ${\Lambda= \OO_S^\times}$. In particular,
${\sqrt\Lambda \cap K^\times =\Lambda}$.

Let~$a$ and~$b$ be as in~\eqref{exayb} and ${\delta:=\gcd(a,b)}$. Then 
$$
(x-1)^ {a/\delta}(y-1)^{b/\delta}\in  \sqrt\Lambda \cap K^\times =\Lambda. 
$$ 
Hence we may assume that~$a$ and~$b$ are coprime. It follows that 
$$
x-1\in (K^\times)^{|b|}\Lambda, \qquad y-1\in (K^\times)^{|a|}\Lambda. 
$$
If ${|b|>1 }$ then $\height(x)$ is bounded by Theorem~\ref{thppz}, and Proposition~\ref{prhxhyb} bounds $\height(y)$. Similarly, both heights are bounded if ${|a|>1}$.

We are left with the case ${|a|=|b |=1}$. Without loss of generality we may assume that ${a=b=1}$ or ${a=1}$,~${b=-1}$. In the former case ${(x-1) (y-1)\in \OO_S^\times}$, which implies that both ${x-1}$ and ${y-1}$ are $S$-units. Hence their heights are bounded by Theorem~\ref{thunits}.

Now assume that ${a=1}$ and ${b=-1}$. Then ${x-1=\gamma(y-1)}$, where ${\gamma \in \Lambda}$. Hence 
\begin{equation}
\label{esfexy}
(x,-1,-\gamma y, \gamma)
\end{equation}
is a solution of the equation ${z_1+\cdots+z_4=0}$ in ${z_1, \ldots, z_4\in \OO_{S}^\times}$. Theorem~\ref{thexp} implies that this equation has, up to equivalence, at most finitely  primitive solutions. Since distinct solutions with ${z_2=-1}$ are not equivalent, there can be at most finitely many primitive solutions of the shape~\eqref{esfexy}.

Now assume that~\eqref{esfexy} is a non-primitive solution. Then we have one of the following:
\begin{alignat}2
\label{exyone}
&x-1=0, \qquad &&-\gamma y+\gamma=0;\\
\label{ex=y}
&x-\gamma y=0, \qquad &&-1+\gamma=0;\\
\label{exy=1}
&x+\gamma=0, \qquad &&-1-\gamma y=0. 
\end{alignat}
Here~\eqref{exyone} is impossible because ${x,y\ne 1}$, in case~\eqref{ex=y} we have ${y=x}$ and in case~\eqref{exy=1} we have ${y=x^{-1}}$. 
The theorem is proved. \qed

\subsection{The greatest common divisor of ${x-1}$ and ${y-1}$}
\label{ssgcd}

The proof of Theorem~\ref{thmodgthree} exploits the following property: if ${x,y \in \Lambda}$ are both distinct from~$1$ and ${x\ne y^{\pm1}}$, then ${x-1}$ and ${y-1}$ cannot have big common divisor. In this subsection we recall the deep result of Corvaja and Zannier~\cite{CZ05} for the case when~$x$ and~$y$ are multiplicatively independent. In the case when they are multiplicatively dependent, we also give a suitable statement.

\subsubsection{The independent case}

For ${x,y\in \bar\Q^\times}$ we define the \textit{logarithmic gcd} of~$x$ and~$y$ by 
$$
\lgcd(x,y) =\sum_{v\in M_K}\min\{\height_v(x), \height_v(y)\}, 
$$
where~$K$ is a number field containing both~$x$ and~$y$ and the local heights $\height_v(\cdot)$ are defined in Subsection~\ref{sslocalh}. A straightforward verification shows that for non-zero ${x,y\in \Z}$ we have ${\lgcd(x,y) =\log\gcd(x,y)}$, which justifies the terminology. 

It is not hard to express ${\lgcd(x,y)}$ in terms of the ``global'' heights: 
$$
\lgcd(x,y):= \height(x,y) - \height(x/y).
$$ 
However, the definition using the local heights is more convenient for our purposes.

The following fundamental result is due to Corvaja and Zannier 
\cite[Corollary~1]{CZ05}.

\begin{theorem}[Corvaja, Zannier]
\label{thcorza}
Let~$\Lambda$ be a finitely generated subgroup of $\bar\Q^\times$, and ${\eps>0}$. Then, for multiplicatively independent ${x,y\in \Lambda}$, we have 
$$
\lgcd(x-1,y-1)\le \eps\max\{\height(x), \height(y)\} +O_\eps(1), 
$$
where the implied constant depends on~$\Lambda$ and~$\eps$.  
\end{theorem}

This theorem is non-effective: its proof is based on the subspace theorem. 

\subsubsection{The dependent case}

We need an analogue of Theorem~\ref{thcorza} for multiplicatively dependent~$x$ and~$y$. It will be given in Proposition~\ref{prmultdep} below, after some preparation. 

If ${x,y\in \bar\Q^\times}$ are multiplicatively dependent, then   either the multiplicative group $\langle x,y\rangle$  is an infinite cyclic group or it  contains a non-trivial root of unity.
If $\langle x,y\rangle$  is  infinite cyclic and~$t$ its generator, then  ${x=t^m}$ and ${y=t^n}$, where~$m$ and~$n$ are coprime integers. When $\langle x,y\rangle$ contains a non-trivial root of unity, a similar statement is given in the following lemma. 

\begin{lemma}
\label{letonoto}
Let~$x$ and~$y$ be multiplicatively dependent non-zero algebraic numbers, both not roots of unity.  Assume that  $\langle x,y\rangle$ contains a root of unity distinct from~$1$. Then  there exist ${t\in \langle x,y\rangle}$ and non-zero coprime integers ${m,n}$ such that 
\begin{equation}
\label{etorsionmn}
x=t^m, \qquad y=\xi t^n,
\end{equation}
where~$\xi$ a root of unity satisfying ${\xi^m\ne 1}$. 
Symmetrically, there exists ${u\in \langle x,y\rangle}$ such that
${x=\theta u^m}$ and ${y=u^n}$ 
(with the same $m,n$ as in~\eqref{etorsionmn}), 
where~$\theta$ is a root of unity satisfying ${\theta^n\ne 1}$. 
\end{lemma}

\begin{proof}
Let~$t$ be a generator of the maximal infinite cyclic subgroup of $\langle x,y\rangle$ containing~$x$. Then we  have~\eqref{etorsionmn} with nonzero coprime integers~$m$ and~$n$ and some root of unity~$\xi$.  
Let ${k\in \Z}$ be such that ${nk\equiv 1\pmod m}$. If  ${\xi^m=1}$ then  ${x=(\xi^kt)^m}$ and ${y=(\xi^kt)^n}$. This means that  ${\langle x,y\rangle=\langle\xi^kt\rangle}$ is an infinite cyclic group, contradicting the hypothesis. Hence ${\xi^m\ne1}$. 

Symmetrically, there exist ${u\in \langle x,y\rangle}$ and non-zero coprime integers ${m',n'}$ such that 
${x=\theta u^{m'}}$ and  ${y=u^{n'}}$, 
where~$\theta$ a root of unity satisfying ${\theta^{n'}\ne 1}$.  Since each of~$t$ and~$u$ generates a maximal subgroup of ${\langle x,y\rangle}$, one of the quotients  $t/u$ or $t/u^{-1}$ is a root of unity, and we may assume that it is $t/u$. This implies that ${m'=m}$ and ${n'=n}$. 
\end{proof}

The following statement is well-known, we omit the proof.

\begin{lemma}
\label{luniuniuni}
Let~$k$ be a positive integer and $\zeta_k$ a primitive root of unity of order~$k$. If ${k=p^e}$ is a prime power, then ${1-\zeta_k\mid p}$. If~$k$ has at least two distinct prime divisors, then ${1-\zeta_k}$ is a unit. 
\end{lemma}

When~$x$ and~$y$ are multiplicatively dependent, is more convenient to use the ``usual'' $\gcd$ in a suitable ring of $S$-integers rather than  the logarithmic $\gcd$. If~$K$ is a number field and~$S$ a finite subset of $M_K$, then for ${x,y\in \OO_S}$ we denote by $\gcds(x,y)$ the ideal of~$\OO_S$ generated by~$x$ and~$y$. When it is a principal ideal $(a)$ then we slightly abuse the notation by writing ${\gcds(x,y)=a}$. 

The absolute norm of a non-zero ideal~$\II$ of~$\OO_S$ is defined  by ${\NN(\II):=\#\OO_S/\II}$. A straightforward verification shows that
\begin{equation*}
\log\NN\bigl( \gcds(x,y)\bigr)=d\sum_{v\in M_K\smallsetminus S}\min \{\height_v(x), \height_v(y)\}, \qquad d:=[K:\Q].  
\end{equation*}

\begin{proposition}
\label{prgcdsres}
For ${x,y\in \OO_S}$ and  coprime positive integers~$m$,~$n$ we have\footnote{When ${x=1}$ we understand that ${(x^m-1)/(x-1):=m}$.} 
$$
\gcds\left(x-1, \frac{y^n-1}{y-1}\right) \ \text{divides} \  n , \qquad \gcds\left(\frac{x^m-1}{x-1}, \frac{y^n-1}{y-1}\right)=1. 
$$
\end{proposition}

\begin{proof}
This follows from the corresponding results about resultants: the resultant of the polynomials ${T-1}$ and ${(T^n-1)/(T-1)}$ is, trivially,~$n$, while that of polynomials ${(T^m-1)/(T-1)}$ and ${(T^n-1)/(T-1)}$ is $\pm1$. Indeed, this resultant is a product of factors of the form ${\zeta_a-\zeta_b}$ where ${a,b>1}$ are coprime integers. Each of these factors is a unit  by Lemma~\ref{luniuniuni}. 
\end{proof}

Now we are ready to state an analogue of Theorem~\ref{thcorza} for ~$x$ and~$y$.

\begin{proposition}
\label{prmultdep}
Let~$K$ be a number field.  We denote by~$\mu$ the order of the group of roots of unity from~$K$. Let~$S$ be a finite subset of~$M_K$ such that ${\mu\in \OO_S^\times}$.  Then for multiplicatively dependent ${x,y\in \OO_S^\times}$, both distinct from~$1$, the following holds.

\begin{enumerate}
\item
\label{it-one}
If  $\langle x,y\rangle$   is an infinite cyclic group   with generator~$t$  then 
$$
\gcds(x-1,y-1)=t-1.
$$

\item 
\label{ione}
If $\langle x,y\rangle$ contains a root of unity distinct from~$1$
 then 
$$
\gcds(x-1,y-1)=1.
$$
\end{enumerate}
\end{proposition}

\begin{proof}
Item~\ref{it-one} is easy. We have
${x=t^m}$ and ${y=t^n}$ with non-zero coprime integers~$m$ and~$n$.  
For any non-zero integer~$k$ we have 
${(t^k-1)/(t^{|k|}-1) \in \OO_S^\times}$.  
Hence 
$$
\gcds(x-1,y-1)=(t-1) \gcds\left(\frac{t^{|m|}-1}{t-1}, \frac{t^{|n|}-1}{t-1}\right),    
$$
and the second factor above is~$1$ by Proposition~\ref{prgcdsres}. This proves item~\ref{it-one}. 

Item~\ref{ione} is more delicate. If~$x$ is a root of unity, then ${x-1\mid \mu}$. Since ${\mu \in\OO_S^\times}$, this implies that ${\gcds(x-1,y-1)=1}$.

From now on we assume that none of~$x$ and~$y$ is a root of unity. 
Let $m,n,t,\xi$  be as in Lemma~\ref{letonoto}:
$$
x=t^m, \qquad y=\xi t^n,\qquad \xi^m\ne1. 
$$
Let~$\ell$ be the order of the root of unity~$\xi$. Note that the polynomials ${\xi T^n-1}$ and ${T^m-1}$ cannot have a common root. Indeed, if~$\eta$ is such a common root then ${\eta^{-n}=\xi}$ and ${\eta^m=1}$. Since ${\gcd(m,n)=1}$, this implies that ${\xi^m=1}$, a contradiction.

We set
$$
\delta:=\gcd(m,\ell), \qquad m':=\frac{|m|}{\delta}, \qquad \ell':=\frac{\ell}{\delta}, \qquad z:=t^\delta. 
$$
We claim that the polynomial ${\xi T^{|n|}-1}$ divides ${(T^{\ell |n|}-1)/(T^\delta-1)}$ in the ring $\Z[\xi][T]$. Indeed, ${\xi T^{|n|}-1\mid T^{\ell |n|}-1}$, and ${\xi T^{|n|}-1}$ cannot have common roots with ${T^\delta-1}$ (because the  latter divides ${T^m-1}$)  which proves our claim.  Thus,  ${y-1}$ divides ${(z^{\ell'|n|}-1)/(z-1)}$ in the ring $\OO_S$. It follows that ${\gcds(x-1, y-1)}$ divides the product 
$$
\gcds\left (z-1, \frac{z^{\ell'|n|}-1}{z-1}, \right) \gcds\left (\frac{z^{m'}-1}{z-1}, \frac{z^{\ell'|n|}-1}{z-1}, \right). 
$$ 
Here the first factor divides $\ell'n$, and the second  factor  is~$1$, again by Proposition~\ref{prgcdsres}. Since ${\ell'\mid \ell\mid \mu}$ by the definition of~$\mu$, and  ${\mu\in \OO_S^\times}$, it follows that ${\gcds(x-1,y-1) \mid n}$. 

By symmetry (see again Lemma~\ref{letonoto}), we have
$$
x=\theta u^m, \qquad y=u^n, \qquad \theta^n\ne 1
$$ 
with the same~$m$ and~$n$. Arguing as above, we prove that ${\gcds(x-1,y-1) \mid m}$. Hence ${\gcds(x-1,y-1)=1}$.
 \end{proof}

\subsection{Primitive divisors}
\label{ssprim}

Another indispensable  tool in the proof of Theorem~\ref{thmodgthree} is Schinzel's theory of primitive divisors. In this subsection we collect the results about primitive divisors used in the proof of Proposition~\ref{prmultdepen}, an important step in  the proof of Theorem~\ref{thmodgthree}. 

Let~$K$ be a number field and let ${\gamma\in K^\times}$ be not a root of unity. We consider the sequence $(u_k)_{k\ge 1}$ defined by ${u_k=:\gamma^k-1}$. A finite prime ${v \in M_K}$ with ${|\gamma|_v=1}$ is called a \textit{primitive divisor} of the term $u_n$ if 
$$
|u_n|_v<1, \qquad |u_k|_v  =1 \quad (k=1, \ldots, n-1). 
$$

Denote by $\Phi_n(T)$ the $n$th cyclotomic polynomial. Recall that, for a finite ${ v \in M_K}$, we denote by $\NN v$ the absolute norm of the corresponding prime of~$K$, see Subsection~\ref{ssabc}. If $\OO_v$ is the local ring of~$v$ in~$K$ and $\gerp_v$ is the maximal ideal of $\OO_v$, then ${\NN v=\#\OO_v/\gerp_v}$. 
The following properties are well-known, we omit the proof.

\begin{proposition}
\label{prprprdiv}
Let~$m$ be a positive integer and~$v$  a primitive divisor of $u_m$. Then ${|\Phi_m(\gamma)|_v <1}$, and the image of~$\gamma$ in the cyclic group ${(\OO_v/\gerp_v)^\times}$ is of order~$m$. In particular, ${m\mid \NN v-1}$,  and ${m\mid n}$
for any positive integer~$n$ with 
 ${|u_n|_v  <1}$.

\end{proposition}

The following fundamental theorem is, essentially, due to Schinzel~\cite{Sc74}; see also~\cite{St77} and \cite[Theorem~4.2]{BL21}. 

\begin{theorem}[Schinzel]
\label{thprim}
Let~$K$ be a number field of degree~$d$ and ${\gamma\in K^\times}$ not a root of unity. Define the sequence $(u_k)_{k\ge 1}$ by ${u_k:=\gamma^k-1}$. If the term $u_n$ does not admit a primitive divisor, then~$n$ is effectively bounded in terms of~$d$. 
\end{theorem}

Let~$S$ be a finite subset of $M_K$. Proposition~\ref{prprprdiv} implies that, in the set-up of Theorem~\ref{thprim},  if $u_n$ has a primitive divisor ${v\in S}$ then ${n \le \max\{\NN v-1: v\in S\}}$. We obtain the following consequence. 

\begin{corollary}
\label{coprim}
In the set-up of Theorem~\ref{thprim}, let~$S$ be a finite subset of $M_K$. If $u_n$ does not admit a primitive divisor ${v\notin S}$, then~$n$ is effectively bounded in terms of~$d$ and~$S$.  
\end{corollary}

If ${\gamma\in \OO_S}$ then one can also bound $\height(\gamma)$. 


\begin{proposition}
\label{prprdiveverywhere}
Let~$K$ be a number field of degree~$d$ and~$S$ a finite subset of~$M_K$.  
Let ${\gamma\in \OO_S^\times}$  and let $(u_k)$ be as in Theorem~\ref{thprim}.   If there exists ${n\ge 1}$ such that $u_n$ does not have a primitive divisor ${v \notin S}$, then $\height(\gamma)$ is effectively bounded in terms of~$d$ and~$S$. 
\end{proposition}

\begin{proof}[Proof of Proposition~\ref{prprdiveverywhere}]
Let~$n$ be such that $u_n$ does not have a primitive divisor ${v\notin S}$. Then ${n\ll 1}$ by Corollary~\ref{coprim}, where the implied constants in this proof depend on~$S$ and~$d$. We denote
${S':= S\cup \{v\in M_K: v\mid n\}}$. 
We claim that 
\begin{equation}
\label{ecycbig}
|\Phi_n(\gamma)|_v=1 \qquad (v\notin S').
\end{equation}
Thus, fix ${v\notin S'}$. If ${|\gamma^n-1|_v=1}$ then there is nothing to prove. 
Now assume that ${|\gamma^n-1|_v <1}$. Since~$v$ is not a primitive divisor, there exists ${m<n }$ such that ${|\gamma^m-1|_v <1}$. If  ${|\Phi_n(\gamma)|_v<1}$, then ${|R|_v <1}$, where we denote by~$R$ the resultant of the polynomials 
$\Phi_n(T)$ and ${T^m-1}$. 
If ${\xi, \xi'}$ are distinct roots of  unity of order dividing~$n$ then ${\xi-\xi' \mid n}$. This implies that ${R\mid n^{n}}$. Hence, ${|\Phi_n(\gamma)|_v<1}$ implies that ${v\mid n}$, which contradicts the assumption  ${v\in S'}$. This proves~\eqref{ecycbig}.

One can rewrite~\eqref{ecycbig} as
\begin{equation}
\label{evnotins}
\height_v(\gamma^n-1) = \height_v\bigl(\Psi_n(\gamma)\bigr) \qquad (v\notin S'), 
\end{equation}
where 
${\Psi_n(T):= (T^n-1)/\Phi_n(T)}$. 
Also, using~\eqref{ecobaker} with ${x:=\gamma^n}$ and remembering that ${n\ll 1}$, we obtain 
$$
\sum_{v\in S'}\height_v(\gamma^n-1) \ll \log^\ast \height(\gamma). 
$$
Combining this with~\eqref{evnotins}, we obtain 
$$
\height(\gamma^n-1) \le \height\bigl(\Psi_n(\gamma)\bigr)  + O(\log^\ast \height(\gamma)). 
$$
Now we are going to apply Proposition~\ref{prhpolx} with  ${x:=\gamma}$, with ${F(X):=X^n-1}$ on the left, and with ${F(X):=\Psi_n(X)}$ on the right. Noting that ${\deg\Psi_n=n-\ph(n)}$ (where $\ph(\cdot)$ is Euler's totient), we obtain 
$$
n\height(\gamma) \le \bigl(n-\ph(n)\bigr)\height(\gamma) +O(\log^\ast \height(\gamma)). 
$$
This bounds $\height(\gamma)$. 
\end{proof}

\subsection{Proof of Theorem~\ref{thmodgthree}}
\label{ssproofmodg}
In this subsection we prove Theorem~\ref{thmodgthree}. By the hypothesis, there exist integers ${a,b,c}$, not all~$0$,   such that 
\begin{equation}
\label{exaybzclam}
(x-1)^a(y-1)^b(z-1)^c\in  \Lambda. 
\end{equation}
Moreover, all of the exponents~$a$,~$b$ and~$c$ are not~$0$, because any two of the numbers ${x-1,\ y-1, \ z-1}$ are multiplicatively independent modulo~$\Lambda$. For the same reason we have
\begin{align}
\label{enotinsqrt}
& x-1,\ y-1,\ z-1 \notin \sqrt\Lambda, \\
\label{enexpmone}
& y\ne x^{\pm1}, \quad z\ne x^{\pm1}, \quad z\ne y^{\pm1}.
\end{align}

\begin{remark}
\label{reproportional}
For the same reason,  if $(a',b',c')$ is another triple of integers such that ${(x-1)^{a'}(y-1)^{b'}(z-1)^{c'}\in  \Lambda}$, then $(a',b',c')$ is proportional to $(a,b,c)$; that is, ${a'/a=b'/b=c'/c}$. 
\end{remark}

\subsubsection{Preparations for the proof of Theorem~\ref{thmodgthree}}

We start from two observation similar to those in the beginning of Subsection~\ref{ssmodgtwo}.  First of all, if we bound the height of one of the numbers $x,y,z$ then the bound the other two as well.

\begin{proposition}
\label{prhxhybthree}
Let ${A>0}$. In the set-up of Theorem~\ref{thmodgthree}, assume that ${\height(x) \le A}$. Then $\height(y)$ and $\height(z)$ are bounded in terms of~$A$ and~$\Lambda$. 
\end{proposition}

\begin{proof}
As in the proof of Proposition~\ref{prhxhyb}, we define~$\Lambda'$ as the multiplicative group generated by~$\Lambda$ and by all  ${\alpha \in K^\times}$ satisfying ${\height(\alpha) A+\log 2}$. We again have  ${x-1\in \Lambda'}$, and~\eqref{exaybzclam} implies that ${y-1}$ and ${z-1}$ are multiplicatively dependent modulo~$\Lambda'$. Since ${b,c\ne 0}$, either both ${y-1}$ and ${z-1}$ belong to $\sqrt{\Lambda'}$, or both do not. In the former case both $\height(y)$ and $\height(z)$ are bounded by Theorem~\ref{thmodgone}; in the latter case they are bounded by Theorem~\ref{thmodgtwo}, because ${z\ne y^{\pm1}}$, see~\eqref{enexpmone}. 
\end{proof}

The second observation is that we can replace~$\Lambda$ by any bigger finitely generated group. 

\begin{proposition}
\label{prreplacethree}
Let $\Lambda'$ be a finitely generated subgroup of $K^\times$ containing~$\Lambda$. 
Assume that Theorem~\ref{thmodgthree} holds with~$\Lambda'$ instead of~$\Lambda$. Then it holds with~$\Lambda$ as well. 
\end{proposition}

\begin{proof}
We have to show the following:  if ${x,y,z\in \Lambda}$ satisfy the hypothesis of Theorem~\ref{thmodgthree} with~$\Lambda$, then, with finitely many exceptions, any two of the numbers ${x-1,\ y-1, \ z-1}$ are multiplicatively independent modulo~$\Lambda'$. 

Thus, assume that ${x-1}$ and ${y-1}$ are multiplicatively dependent modulo~$\Lambda'$. Let us show that the heights of $x,y,z$ are bounded. 
If ${x-1\in \sqrt{\Lambda'}}$ then all the three heights are bounded by Theorem~\ref{thmodgone} and Proposition~\ref{prhxhybthree}. If ${x-1, y-1\notin  \sqrt{\Lambda'}}$ then the heights of~$x$ and~$y$ are bounded by Theorem~\ref{thmodgtwo} because ${y\ne x^{\pm1}}$, and $\height(z)$ is bounded by Proposition~\ref{prhxhybthree}.
\end{proof}

Let~$\mu$ be the order of the group of roots of unity of~$K$, and~$S$ a finite subset of $M_K$ such that ${\Lambda\le \OO_S^\times}$ and 
\begin{equation}
\label{emuinost}
\mu \in \OO_S^\times. 
\end{equation}
Later, we will impose one more condition on~$S$. In the sequel we will assume that ${\Lambda=\OO_S^\times}$;  that is, ${x,y,z\in \OO_S^\times}$ and there exist non-zero integers $a,b,c$ such that
\begin{equation}
\label{econdmdep}
(x-1)^a(y-1)^b(z-1)^{c}\in \OO_S^\times. 
\end{equation}
Note that we may assume that 
\begin{equation}
\label{ecoprimeabc}
\gcd(a,b,c)=1,
\end{equation}
because ${\sqrt{\OO_S^\times}\cap K^\times=\OO_S^\times}$, and for the same reason~\eqref{enotinsqrt} becomes 
\begin{equation}
\label{enotinost}
x-1,\ y-1,\ z-1 \notin\OO_S^\times. 
\end{equation}

Denote by $T_x$ the set of ${v \in M_K\smallsetminus S}$ such that ${|x-1|_v <1}$ (or, equivalently, such that ${\height_v(x-1)>0}$). Similarly, we define $T_y$ and $T_z$. Note that
\begin{equation}
\label{enonempty}
T_x,T_y,T_z \ne \varnothing; 
\end{equation}
indeed, if ${T_x=\varnothing }$ then ${x-1\in \OO_S^\times}$, contradicting~\eqref{enotinost}.

Without loss of generality we may assume that either ${a,b, c>0}$ or 
\begin{equation}
\label{eabpcn}
a,b >0, \quad c<0.
\end{equation}
If ${a,b, c>0}$ then~\eqref{econdmdep} implies that each of ${x-1}$, ${y-1}$ and  ${z-1}$ is an $S$-unit, contradicting~\eqref{enotinost}. Hence we have~\eqref{eabpcn}, and we can rewrite~\eqref{econdmdep} as 
\begin{equation}
\label{ehard}
(x-1)^a (y-1)^{b}(z-1)^{-|c|} \in \OO_S^\times.  
\end{equation}
The fundamental consequence of~\eqref{ehard} is that 
\begin{equation}
\label{esets}
T_x \cup T_y = T_z. 
\end{equation} 
Indeed, if ${v\notin S}$ then we have both ${|x-1|_v\le 1}$ and ${|y-1|_v \le 1}$. Hence~\eqref{ehard} implies that 
${|z-1|_v<1}$ if and only if  $|x-1|_v<1$ or $|y-1|_v<1$.

The following property will be repeatedly used. 

\begin{proposition}
\label{prmultdepen}
Assume that~$x$ and~$z$ are multiplicatively dependent. Then, with finitely many exceptions,  ${z=x^n}$ for some  integer~$n$ satisfying ${|n|\ge 2}$. 
\end{proposition}
Similarly, if~$y$ and~$z$ are multiplicatively dependent, then, with finitely many exceptions,  ${z=y^r}$,  where ${|r|\ge2}$. 

\begin{proof}
By~\eqref{emuinost}, Proposition~\ref{prmultdep} applies. 
Hence, if the group ${\langle x,z\rangle}$ contains a non-trivial root of unity then we have ${\gcd_S(x-1,z-1) =1}$. It follows that ${T_x=\varnothing}$, which contradicts~\eqref{enonempty}. 

Thus, ${\langle x,z\rangle}$ is an infinite cyclic group. Let~$t$ be its generator. Then ${x=t^m}$ and ${z=t^n}$ for non-zero coprime  ${m,n\in \Z}$, and we may assume that  ${m>0}$. We are going to apply the results of Subsection~\ref{ssprim} to the sequence $(u_k)$ defined by ${u_k:=t^k-1}$.

If the term $u_m$ does not have a primitive divisor ${v\notin S}$, then  both~$m$ and $\height(t)$ are bounded, see Theorem~\ref{thprim} and Proposition~\ref{prprdiveverywhere}. Hence $\height(x)$ is bounded, and   Proposition~\ref{prhxhybthree} bounds the other two heights.

Now assume that the term $u_m$ has a primitive divisor ${v\notin S}$. Then ${v\in T_x}$, which implies that ${v\in T_z}$, that is, 
${|t^n-1|_v <1}$. It follows that ${|t^{|n|}-1|_v <1}$, which implies that~$m$ divides $|n|$, see Proposition~\ref{prprprdiv}. Since~$m$ and~$n$ are coprime, we must have ${m=1}$, which means that ${z=x^n}$. We cannot have ${|n|=1}$ by~\eqref{enexpmone}. Hence ${|n|\ge2}$.  
\end{proof}

We denote ${d:=[K:\Q]}$ and define
${\Sigma_x :=d^{-1}\sum_{v\in T_x} \log \NN v}$. 
We define similarly  $\Sigma_y$ and $\Sigma_z$.

\begin{proposition}
\label{prsigmas}
\begin{enumerate}
\item
\label{isxplussy}
We have ${\Sigma_z\le \Sigma_x+\Sigma_y}$.

\item
\label{isxlehx}
We have ${\Sigma_x \le \height(x-1)}$, and similarly for $\Sigma_y$ and $\Sigma_z$.

\item
\label{isxlelgcd}
We have ${\Sigma_x \le \lgcd(x-1, z-1)}$, and similarly for $\Sigma_y$. 

\item
\label{iconxlesx}
We have 
${\cond(x)+\cond(x^{-1})+\cond(x-1) \le  \Sigma_x +O(1)}$, where the implied constant depends on~$S$, and similarly for~$y$ and~$z$. 
\end{enumerate}
\end{proposition}

Here $\cond(\cdot)$ denotes the conductor, introduced in Subsection~\ref{ssabc}.

\begin{proof}
Item~\ref{isxplussy} follows immediately from~\eqref{esets}. To prove item~\ref{isxlehx} note that ${\height_v(x-1) >0}$ for every ${v\in T_x}$, which implies that ${\height_v(x-1) \ge d^{-1}\log \NN v}$. It follows that
${\Sigma_x \le \sum_{v\in T_x}\height_v(x-1) \le \height(x-1)}$.

To prove item~\ref{isxlelgcd}, note that, since ${T_x \subset T_z}$, we have ${\height_v(z-1) \ge d^{-1}\log \NN v}$ for every ${v\in T_x}$. Hence 
$$
\Sigma_x \le \sum_{v\in T_x}\min\{\height_v(x-1), \height_v(z-1)\} \le \lgcd(x-1, z-1). 
$$
To prove item~\ref{iconxlesx}, note that, since ${x\in \OO_S^\times}$, we  have 
${\cond(x) , \cond(x^{-1}) \ll1}$. 
Also, 
\begin{equation*}
\cond(x-1)\le  d^{-1}\left(\sum_{v\in T_x} +\sum_{v\in S}\right)\log\NN v =\Sigma_x +O(1). 
\end{equation*}
In both cases the implied constant depends on~$S$. Summing up, we are done. 
\end{proof}

From now on, we are going to assume the validity of Conjecture~\ref{cowabc}.  Let~$\kappa$ be as in Conjecture~\ref{cowabc}. The following statement is crucial. 

\begin{proposition}
\label{prnorr}
With finitely many exceptions, one of the following holds:  ${z=x^n}$ with ${2\le |n|\le 3\kappa}$ or ${z=y^r}$ with ${2\le |r|\le 3\kappa}$. 
\end{proposition}

\begin{proof}
By Proposition~\ref{prmultdepen},  we have to rule out, with finitely many exceptions, each of the following three cases:
\begin{align}
\label{eindep}
&\text{$z$ is multiplicatively independent with~$x$ and with~$y$};\\
\label{emixed}
&\text{$z$ is multiplicatively independent with~$x$ and ${z=y^r}$ with ${|r|\ge 3\kappa}$}\\
\label{edep}
&\text{${z=x^n=y^r}$ with ${|n|,|r|\ge 3\kappa}$}. 
\end{align} 
Fix ${\eps>0}$, to be specified later. Assume~\eqref{eindep}. If ${\height(z) \ge \height(x), \height(y)}$ then  Theorem~\ref{thcorza} implies that 
$$
\lgcd(x-1,z-1) \le \eps \height(z) +O_\eps(1), \qquad \lgcd(y-1,z-1) \le \eps \height(z) +O_\eps(1), 
$$
where the implied constants depend on~$\eps$ and~$S$.  Proposition~\ref{prsigmas} now implies that 
$$
\cond(z)+\cond(z^{-1})+\cond(z-1) \le 2\eps \height(z)+ O_\eps(1). 
$$
Using Conjecture~\ref{cowabc}, we obtain 
${\height(z) \le 2\kappa\eps \height(z)+O_\eps(1)}$. Specifying ${\eps := 1/3\kappa}$, we bound $\height(z)$. 

If, say, ${\height(x) \ge \height(z)}$ then ${\lgcd(x-1,z-1) \le \eps \height(x) +O_\eps(1)}$. We obtain 
\begin{equation}
\label{econxb}
\cond(x)+\cond(x^{-1})+\cond(x-1) \le \eps \height(x)+ O_\eps(1), 
\end{equation}
and we bound $\height(x)$ using Conjecture~\ref{cowabc}. Thus,~\eqref{eindep} may happen only for finitely many triples $(x,y,z)$.

Now assume~\eqref{emixed}. If ${\height(x) \ge \height(z)}$ then we again have~\eqref{econxb} and we are done. If ${\height(z) \ge \height(x)}$ then ${\lgcd(x-1,z-1) \le \eps \height(z) +O_\eps(1)}$. Also, 
$$
\Sigma_y \le \height(y-1) \le \height(y) + 1 = \frac{1}{|r|}\height(z)+1 \le \frac{1}{3\kappa}\height(z)+1. 
$$
We obtain 
$$
\cond(z)+\cond(z^{-1})+\cond(z-1) \le \left(\eps+ \frac{1}{3\kappa}\right)\height(z)+ O_\eps(1). 
$$
Conjecture~\ref{cowabc} now gives ${\height(z) \le (\kappa\eps+1/3) \height(z)+O_\eps(1)}$, and   we bound $\height(z)$ specifying ${\eps := 1/3\kappa}$. 

Finally let us assume~\eqref{edep}. Then 
$$
\Sigma_x \le \frac{1}{3\kappa}\height(z)+1, \qquad \Sigma_y \le \frac{1}{3\kappa}\height(z)+1. 
$$
Hence
$$
\cond(z)+\cond(z^{-1})+\cond(z-1) \le \frac{2}{3\kappa} \height(z)+ 2, 
$$
and  Conjecture~\ref{cowabc} implies that ${\height(z) \le (2/3) \height(z)+3\kappa}$, which bounds $\height(z)$. 
\end{proof}

\subsubsection{Completing the proof of Theorem~\ref{thmodgthree}}

Now we are ready to complete the proof of Theorem~\ref{thmodgthree}. So far, the only condition imposed on the set~$S$ was~\eqref{emuinost}.  From now on, we will also assume that~$S$ contains all~$v$ with underlying prime ${p\le 3\kappa}$; equivalently,
\begin{equation}
\label{enrinost}
m\in \OO_S^\times \qquad (m\in \Z, \quad 2\le |m|\le 3\kappa). 
\end{equation}

By Proposition~\ref{prnorr},  we may assume that ${z=y^r}$ with ${2\le |r|\le 3\kappa}$.  Using ${(y^r-1)/(y^{|r|}-1)\in \OO_S^\times}$, we  rewrite~\eqref{econdmdep} as
\begin{equation}
\label{econdzyr}
(x-1)^a(y-1)^b\bigl(y^{|r|}-1\bigr)^{c}\in \OO_S^\times. 
\end{equation} 
Since ${r\in \OO_S^\times}$ by~\eqref{enrinost}, Proposition~\ref{prgcdsres} implies that
\begin{equation}
\label{egcdsyyr}
\gcds\left(y-1, \frac{y^{|r|}-1}{y-1}\right)=1. 
\end{equation}
The proof of Theorem~\ref{thmodgthree}  splits into three cases, depending on the sign of ${b+c}$. 

\subsubsection*{The case ${b+c>0}$}

Assume first that ${b+c>0}$, and let us show that this is impossible. We can rewrite~\eqref{econdzyr} as 
$$
(x-1)^a(y-1)^{b+c}\left(\frac{y^{|r|}-1}{y-1}\right)^{-|c|}\in \OO_S^\times.
$$
Since ${b+c>0}$, this implies that ${y-1}$ divides ${\bigl((y^{|r|}-1)/(y-1)\bigr)^{|c|}}$ in the ring~$\OO_S$. From~\eqref{egcdsyyr} we conclude that ${y-1\in \OO_S^\times}$, which contradicts~\eqref{enotinost}.  Thus, we cannot have ${b+c>0}$.

\subsubsection*{The case ${b+c=0}$}

Now let us assume that ${b+c=0}$. We will prove that in this case, with finitely many exceptions, we have
\begin{equation}
\label{econclusion}
y=-x^{\pm1}, \quad z=y^{\pm2}. 
\end{equation}
When ${b+c=0}$, equation~\eqref{econdzyr} becomes 
\begin{equation}
\label{ebplusc=zero}
(x-1)^a\left(\frac{y^{|r|}-1}{y-1}\right)^{-b}\in \OO_S^\times.
\end{equation}
We have ${\gcd(a,b)=1}$ by~\eqref{ecoprimeabc}. It follows that 
$$
x-1 \in (K^\times)^b\OO_S^\times, \qquad \frac{y^{|r|}-1}{y-1} \in (K^\times)^a\OO_S^\times. 
$$
Now, if ${b>1}$ then $\height(x)$ is bounded by Theorem~\ref{thppz}, and if ${a>1}$ then $\height(y)$ is bounded by Corollary~\ref{coratpow}. 

Thus, with finitely many exceptions we have  
\begin{equation}
\label{eoneonemone}
a=b=1, \qquad c=-1. 
\end{equation}
In this case~\eqref{ebplusc=zero} can be rewritten as 
\begin{equation}
\label{esixtermeq}
(x-1)(y-1)=\gamma(y^{|r|}-1), 
\end{equation}
where ${\gamma \in \OO_S}$. It follows that 
\begin{equation}
\label{esolsix}
(xy,\ -x,\ -y,\ 1,\ -\gamma y^{|r|},\ \gamma) 
\end{equation}
is a solution of the equation ${z_1+\cdots+z_6=0}$ in ${z_1, \ldots, z_6\in \OO_S^\times}$. Theorem~\ref{thexp} implies that this equation has at most finitely many primitive solutions, up to equivalence. Since distinct solutions with ${z_4=1}$ are not equivalent, there are at most finitely many primitive solutions of the shape~\eqref{esolsix}. 

Now assume that~\eqref{esolsix} is a non-primitive solution. Since ${y^{|r|}\ne 1}$, a proper subset of the first four coordinates of~\eqref{esolsix} must have sum~$0$. None of the the $3$-term sums 
$$
-x-y-1, \quad xy-y-1, \quad xy-x-1, \quad xy-x-y
$$
may be~$0$, because vanishing of any of these sums implies that ${x-1,y-1\in \OO_S^\times}$, contradicting~\eqref{enotinost}. Among the six $2$-term sums, the following four 
$$
xy-x, \quad xy-y, \quad -x+1, \quad -y+1 
$$ 
cannot vanish either, because ${x,y\ne 1}$. Hence ${xy+1=0}$ or ${-x-y=0}$, that is,  ${y=-x^{\pm1}}$. 
Now~\eqref{esixtermeq} implies that
${(y^{|r|}-1)/(y^2-1) \in \OO_S^\times}$. 
When ${|r|\ne2}$, Corollary~\ref{corat} bounds the height of~$y$. This, with finitely many exceptions we have ${|r|=2}$, that is, ${z=y^{\pm2}}$, which, together with ${y=-x^{\pm1}}$,  shows that~\eqref{econclusion} holds with finitely many exceptions, as wanted. 

\subsubsection*{The case ${b+c<0}$}

We will show that ${b+c<0}$ can happen only for finitely many triples $(x,y,z)$ satisfying the hypothesis of Theorem~\ref{thmodgthree}.
This time we rewrite~\eqref{econdzyr} as 
$$
(x-1)^a(y-1)^{-|b+c|}\left(\frac{y^{|r|}-1}{y-1}\right)^{-|c|}\in \OO_S^\times, 
$$
which implies that ${T_y\subset T_x}$, and so ${T_z=T_x}$. We will use now the argument similar to the proof of Proposition~\ref{prnorr} to show that, with finitely many exceptions, 
\begin{equation}
\label{ezxn}
\text{${z=x^n}$ with ${2\le |n|\le 3\kappa}$}. 
\end{equation}
By Proposition~\ref{prmultdepen}, we have to rule out, with finitely many exceptions, each of the following two cases:
\begin{align}
\label{eindepxz}
&\text{$z$ is multiplicatively independent with~$x$};\\
\label{emixedxz}
&\text{${z=x^n}$ with ${|n|\ge 3\kappa}$}. 
\end{align} 
Fix ${\eps>0}$, to be specified later. In case~\eqref{eindepxz} we have
$$
\Sigma_x=\Sigma_z\le \lgcd(x-1,z-1) \le \eps\max\{\height(x), \height(z)\}+O_\eps(1). 
$$
Hence 
$$
\cond(z)+\cond(z^{-1})+\cond(z-1) \le \eps\max\{\height(x), \height(z)\}+O_\eps(1), 
$$
and Conjecture~\ref{cowabc} implies that ${\height(z) \le \kappa\eps\max\{\height(x), \height(z)\}+O_\eps(1)}$. In a similar fashion we bound $\height(x)$. We obtain 
$$
\max\{\height(x), \height(z)\} \le \kappa\eps\max\{\height(x), \height(z)\}+O_\eps(1), 
$$
and specifying ${\eps:=1/3\kappa}$, we bound both heights.

In case~\eqref{emixedxz} we bound 
$$
\Sigma_z=\Sigma_x \le \height(x-1)\le \height(x)+1 \le \frac{1}{3\kappa}\height(z)+1. 
$$
Hence 
$$
\cond(z)+\cond(z^{-1})+\cond(z-1) \le  \frac{1}{3\kappa}\height(z)+ O(1), 
$$
and Conjecture~\ref{cowabc} implies that ${\height(z) \le (1/3) \height(z)+O(1)}$, bounding $\height(z)$. 

Thus, we have~\eqref{ezxn} with finitely many many exception. Now, assuming~\eqref{ezxn}, we may argue as before, but with~$x$ and~$a$ replacing~$y$ and~$b$. We show that ${a+c>0}$ is impossible, and in the case ${a+c=0}$ we must have~\eqref{eoneonemone}, with finitely many exceptions, which is not possible because ${b+c<0}$. Thus, with finitely many exceptions we have ${a+c<0}$, which implies ${T_y\subset T_x}$, and we obtain ${T_y=T_x}$.

Proposition~\ref{prmultdep} implies that ${\langle x,y\rangle}$ is an infinite cyclic group: otherwise  
${\gcds(x-1,y-1)=1}$, which is impossible because ${T_x=T_y\ne \varnothing}$, see~\eqref{enonempty}. Let~$t$ be a generator of  ${\langle x,y\rangle}$; then ${x=t^\ell}$ and ${y=t^q}$, where~$\ell$ and~$q$ are coprime integers. Since~$\ell$ and~$q$ are coprime, ${z=t^{\ell n}=t^{qr}}$ implies that ${\ell\mid r}$ and ${q\mid n}$; in particular, ${q,\ell\le 3\kappa}$.

Since  ${\gcds(x-1,y-1)=t-1}$ (again by Proposition~\ref{prmultdep}), we have 
$$
\frac{t^\ell-1}{t-1}, \frac{t^q-1}{t-1} \in \OO_S^\times. 
$$
If ${|\ell|>1}$ or ${|q|>1}$ then Corollary~\ref{corat} bounds the height of~$t$. Hence, with finitely many exceptions,  we have ${|\ell|=|q|=1}$, that is, ${y=x^{\pm1}}$, which is impossible by~\eqref{enexpmone}. This proves that ${b+c<0}$ may happen only for finitely many triples $(x,y,z)$, as wanted.  

\subsubsection*{Conclusion}
Let us conclude: we showed that~\eqref{econclusion} holds for all but finitely many triples $(x,y,z)$ satisfying the hypothesis of   Theorem~\ref{thmodgthree}.  The theorem is proved. \qed

\subsection{A conjecture}
\label{ssconj}

Let~$\Lambda$ be a finitely generated subgroup of~$\bar\Q^\times$, and ${m\ge 2}$. We assume that ${\zeta_{m-1}\in \Lambda}$. 
In this subsection we give a conjectural description of ${x_1, \ldots, x_m\in \Lambda}$ such that ${x_1-1, \ \ldots, \ x_m-1}$ are multiplicatively dependent modulo~$\Lambda$. 

\begin{conjecture}
\label{comdgm}
Let ${x_1, \ldots, x_m\in \Lambda}$ be distinct from~$1$. Assume that the algebraic numbers ${x_1-1, \ \ldots, \ x_m-1}$ are multiplicatively dependent modulo~$\Lambda$, but any ${m-1}$ of these numbers are multiplicatively independent modulo~$\Lambda$. Then, with finitely many exceptions, after a permutation, we have
$$
x_k=(\zeta_{m-1}^{k-1} x_1)^{\pm1} \quad (k=2, \ldots, m-1), \qquad x_m=x_1^{\pm (m-1)}. 
$$ 
\end{conjecture}
Theorem~\ref{thmodgtwo} confirm this conjecture for ${m=2}$, and Theorem~\ref{thmodgthree} confirms it for ${m=3}$, assuming the weak $abc$. We believe that, generalizing the argument of Subsection~\ref{ssproofmodg}, one can deduce Conjecture~\ref{comdgm} from the weak $abc$, though, probably, some  new ideas might  be required. We also believe that a suitable generalization of the argument of Section~\ref{sproofmthree} shows that Conjecture~\ref{comain} follows from Conjecture~\ref{comdgm}.

\section{Proof of Theorem~\ref{thmthree}}
\label{sproofmthree}


Recall that~$K$ is a number field for which the weak $abc$-conjecture holds true, and~$\Gamma$,~$\Delta$ are finitely generated almost disjoint subgroups of~$K^\times$. We set ${\Lambda:=\sqrt{\langle \Gamma, \Delta, -1\rangle}\cap K^\times}$. It is a finitely generated subgroup of $K^\times$ with the property ${z^m \in \Lambda\Rightarrow z\in \Lambda}$ for ${z \in K^\times}$ and a non-zero integer~$m$.

Let $x_k$ and $y_k$ be as in the statement of Theorem~\ref{thmthree}: the three sums 
$$
x_1+y_1, \quad x_2+y_2, \quad x_3+y_3 
$$
are non-zero and multiplicatively dependent, but any two of them are independent. We define  ${z_k:=-x_k/y_k\in \Lambda}$. 
We claim that~$z_k$ are distinct from~$1$ and from each other: ${z_1,z_2,z_3 \ne 1}$, and 
\begin{equation}
\label{ezkdist}
z_j\ne z_k \qquad (1\le j<k\le 3). 
\end{equation}
Indeed, if ${z_k=1}$, then ${x_k+y_k=0}$, which contradicts the hypothesis of Theorem~\ref{thmthree}. And if, 
say, ${z_1=z_2}$ then ${x_1/x_2=y_1/y_2}$. Since the groups~$\Gamma$ and~$\Delta$ are almost disjoint, this implies that ${x_1/x_2=y_1/y_2}$ is a root of unity. It follows that ${x_1+y_1}$ and ${x_2+y_2}$ are multiplicatively dependent,  again contradicting the hypothesis.

By the hypothesis of Theorem~\ref{thmthree},
there exist non-zero  integers ${m_1,m_2,m_3}$ such that 
\begin{equation}
\label{emultdepsums}
(x_1+y_1)^{m_1}(x_2+y_2)^{m_2}(x_3+y_3)^{m_3}=1. 
\end{equation} 
Then 
\begin{equation}
\label{erelationzs}
(z_1-1)^{m_1}(z_2-1)^{m_2}(z_3-1)^{m_3}\in \Lambda. 
\end{equation}
If two of the numbers ${z_k-1}$ belong to~$\Lambda$ then the the third one also does, because the exponents~$m_k$ in~\eqref{erelationzs} are non-zero. In this case the height of every~$z_k$ is bounded by Theorem~\ref{thunits}, and Corollary~\ref{coadis} bounds the heights of every~$x_k$ and~$y_k$. 

From now on we assume that at most one of the numbers ${z_k-1}$ belongs to~$\Lambda$. After renumbering, we may assume that
\begin{equation}
\label{ezzninl}
z_2-1, \ z_3-1 \notin \Lambda
\end{equation}
If ${z_1-1\in \Lambda}$, then~$z_2$ and~$z_3$ satisfy the hypothesis of Theorem~\ref{thmodgtwo}, because in this case ${(z_2-1)^{m_2}(z_3-1)^{m_3}\in \Lambda}$. This implies that ${z_3=z_2^{\pm1}}$ with finitely many exceptions. But ${z_3\ne z_2}$ by~\eqref{ezkdist}, which shows that in the case ${z_1-1\in \Lambda}$ we have 
\begin{equation}
\label{ezthreeztwomone}
z_3=z_2^{-1}
\end{equation}
with finitely many exceptions. 

Now assume that ${z_1-1\notin \Lambda}$. Then there are two possibilities: either ${z_1,z_2,z_3}$ satisfy  the hypothesis of Theorem~\ref{thmodgthree}, in which case, after renumbering, we have 
\begin{equation}
\label{emc}
z_2 =-z_1^{\pm1}, \qquad z_3=z_1^{\pm2}, 
\end{equation}
with finitely many exceptions, 
or, after renumbering,~$z_2$ and~$z_3$ satisfy the hypothesis of Theorem~\ref{thmodgtwo}, in which case we again have~\eqref{ezthreeztwomone} with finitely many exceptions. 

Thus, with finitely many exceptions, a permutation of $z_1,z_2,z_3$ satisfies either~\eqref{ezthreeztwomone} or~\eqref{emc}.  We consider these cases separately.

\subsection{Case~\eqref{emc}}

Assume~\eqref{emc}. 
If both~$\pm$ are~$+$ then ${z_2=-z_1}$ and ${z_3=z_1^2}$; in other words, ${x_2/y_2=-x_1/y_1}$ and ${x_3/y_3=-(x_1/y_1)^2}$.  Using almost disjointness of~$\Gamma$ and~$\Delta$, we see that 
${\xi_2:=x_2/x_1=-y_2/y_1}$ and ${\xi_3:= x_3/x_1^2=-y_3/y_1^2}$ are roots of unity. This proves that~\eqref{emthree} holds with ${x:=x_1}$ and ${y:=-y_1}$. 

Now let us assume that in~\eqref{emc} we have  ${z_2=-z_1^{-1}}$ and ${z_3=z_2^2}$. Arguing as before, we find that  ${\xi_3:= x_3/x_1^2=-y_3/y_1^2}$ is a root of unity, and so is ${\xi_2:=-x_2x_1=y_2y_1}$. It follows that 
$$
x_2+y_2 = \xi_2\left(-\frac1{x_1}+\frac1{y_1}\right)= \xi_2\frac{x_1-y_1}{x_1y_1}, \qquad x_3+y_3 =x_1^2-y_1^2. 
$$
Substituting this into~\eqref{emultdepsums},  we obtain 
\begin{equation}
\label{ebadcase}
(x_1+y_1)^n(x_1-y_1)^r = \xi(x_1y_1)^s, \qquad  
\end{equation}
where
\begin{equation}
\label{enrs}
n:=m_1+m_3, \ r:= m_2+m_3, \ s:=m_2,
\end{equation}
and~$\xi$ is a root of unity. In the two remaining cases  
${z_2=-z_1, \ z_3=z_2^{-2} }$  and ${ z_2=-z_1^{-1}, \ z_3=z_2^{-2}}$,  
we again obtain~\eqref{ebadcase}, with~$n$ and~$r$ as in~\eqref{enrs}, but with 
$$
s:=2m_3 \quad \text{or}\quad s:=m_2+2m_3,
$$ 
depending on the case. Note that we cannot have ${n=r=s=0}$ in any of the cases, because this would imply ${m_1=m_2=m_3=0}$, a contradiction.

Let~$S$ a finite subset of~$M_K$  such that ${2\in \OO_S^\times}$ and ${\Lambda\le \OO_S^\times}$. We are going to show that~\eqref{ebadcase} with not all of $n,r,s$ zero may happen only for finitely many choices of our variables $x_1$ and~$y_1$. Since the other variable are expressed in terms of $x_1,y_1$ and roots of unity from the  field~$K$, we have finitely many choices for them as well.

If ${n=r=0}$ then ${s\ne 0}$ and $x_1y_1$ is a root of unity. By almost disjointness, both~$x_1$ and~$y_1$ must be roots of unity, which leaves only finitely many choices for $x_1$ and $y_1$.

If ${n>0 }$ and ${r\ge 0}$ then ${x_1+y_1\in \OO_S^\times}$, and the problem is reduces to the  unit equation, as in  the case~\eqref{eredtounit} in the proof of Theorem~\ref{thmtwo}. Thus, we again have only finitely many choices for~$x_1$ and~$y_1$. The case ${n\ge 0}$, ${r>0}$ is similar.

Now assume that~$n$ and~$r$ are non-zero integers of different signs; say, ${n>0}$ and ${r<0}$. Replacing~$r$ by $-r$, we obtain 
\begin{equation}
\label{enrdifs}
(x_1+y_1)^n = \lambda(x_1-y_1)^r
\end{equation}
with ${\lambda \in \OO_S^\times}$ and ${n,r>0}$. Since ${x_1, y_1\in \OO_S^\times}$, for every ${v\in M_K\smallsetminus S}$ we have 
$$
|x_1+y_1|_v, \ |x_1-y_1|_v \le 1. 
$$
If ${v\in M_K\smallsetminus S}$ is such that ${|x_1+y_1|_v <1}$ then ${|x_1-y_1|_v <1}$ by~\eqref{enrdifs}. It follows that
$$
|2x_1|_v=|(x_1+y_1)+(x_1-y_1)|_v<1.
$$ 
Since ${x_1\in \OO_S^\times}$, this implies that ${|2|_v<1}$. Hence ${v \in S}$ by our definition of~$S$, a contradiction. 

Thus, ${|x_1+y_1|_v=1}$ for every ${v\in M_K\smallsetminus S}$, and we again have
${x_1+y_1\in \OO_S^\times}$. As before, this implies that there are at most finitely many possible~$x_1$ and~$y_1$. This completes the proof in  case~\eqref{emc}.

\subsection{Case~\eqref{ezthreeztwomone}} 
In this case ${\xi:=x_2x_3=y_2y_3}$ is a root of unity by almost disjointness. Hence we have
\begin{equation}
\label{etto}
x_3=\xi x_2^{-1}, \qquad y_3=\xi y_2^{-1},
\end{equation}
which is~\eqref{ethreetwomone}. We only have to show that, with finitely many exceptions,  ${x_1+y_1}$  and $x_2y_2$  are multiplicatively dependent. 

Substituting~\eqref{etto} into~\eqref{emultdepsums}, we obtain
\begin{equation*}
(x_1+y_1)^{m_1} (x_2+y_2)^{m_2+m_3}=\xi'(x_2y_2)^{m_2}
\end{equation*}
with some root of unity~$\xi'$. 
If ${m_2+m_3=0}$ then we are done. Let us show that we have ${m_2+m_3=0}$ with finitely many exceptions. 

Assuming that ${m_2+m_3\ne 0}$, we write 
\begin{equation}
\label{emomtt}
(z_1-1)^{m_1}(z_2-1)^{m_2+m_3}\in \Lambda. 
\end{equation}
If ${z_1-1\in \Lambda}$ then ${z_2-1\in \Lambda}$ as well, by~\eqref{emomtt}, which contradicts~\eqref{ezzninl}. Hence  ${z_1-1\notin \Lambda}$, and Theorem~\ref{thmodgtwo} implies that ${z_2=z_1^{\pm1}}$ with finitely many exceptions. However, ${z_2=z_1}$ is impossible by~\eqref{ezkdist}, and ${z_2=z_1^{-1}}$ implies that ${z_1=z_3}$, again impossible. This proves that ${m_2+m_3=0}$ with finitely many exceptions, as wanted. 

Theorem~\ref{thmodgthree} is proved. \qed

{\footnotesize

\bibliographystyle{amsplain}
\bibliography{pillai_mult}

\bigskip

\noindent 
Yuri Bilu: 
Institut de Mathématiques de Bordeaux, Université de Bordeaux \& CNRS, Talence, France; 
\url{yuri@math.u-bordeaux.fr}

\bigskip

\noindent
Florian Luca:
Mathematics Division, Stellenbosch University, Stellenbosch, South Africa and Max-Planck Institute for Software Systems, Saarbr\"ucken, Germany; 
\url{fluca@sun.ac.za}
}

\end{document}